\documentclass[12pt,reqno]{amsart}

\usepackage{amssymb,amsmath,graphicx,amsfonts,euscript}
\usepackage{times}
\usepackage{mathrsfs}
\usepackage{color}
\usepackage{cite}
\allowdisplaybreaks

\newcommand{\om}{\omega}

\newcommand{\na}{\nabla}

\newcommand{\p}{\partial}

\def\om{\omega}

\def\na{\nabla}

\def\om{\omega}

\def\na{\nabla}

\newtheorem{theorem}{\noindent Theorem}[section]
\newtheorem{lemma}{\noindent Lemma}[section]
\newtheorem{proposition}{\noindent Proposition}[section]

\newcommand{\beq}{\begin{equation}}
\newcommand{\eeq}{\end{equation}}
\newcommand{\ben}{\begin{eqnarray}}
\newcommand{\een}{\end{eqnarray}}
\newcommand{\beno}{\begin{eqnarray*}}
\newcommand{\eeno}{\end{eqnarray*}}
\everymath{\displaystyle}

\makeatletter
\@namedef{subjclassname@2020}{%
	\textup{2020} Mathematics Subject Classification}
\makeatother

\numberwithin{equation}{section}

\begin{document}
	
\title[The 3D MHD equations with mixed dissipation]{Global solutions to 3D incompressible MHD system
	with dissipation in only one direction}

\author[Hongxia Lin, Jiahong Wu and Yi Zhu]{Hongxia Lin$^{1}$, Jiahong Wu$^{2}$ and Yi Zhu$^{3}$}

\address{$^1$ College of Mathematics and Physics,  and Geomathematics Key Laboratory of Sichuan Province, Chengdu University of Technology, Chengdu, 610059, P. R. China}

\email{linhongxia18@126.com}

\address{$^2$ Department of Mathematics, Oklahoma State University, Stillwater, OK 74078, United States}

\email{jiahong.wu@okstate.edu}

\address{$^3$ Department of Mathematics, East China University of Science and Technology, Shanghai 200237,  P.R. China}

\email{zhuyim@ecust.edu.cn}

\vskip .2in
\begin{abstract}
The small data global well-posedness of the 3D incompressible Navier-Stokes
equations in $\mathbb R^3$ with only one-directional dissipation remains an outstanding open problem. The
dissipation in just one direction, say $\p_1^2 u$ is simply insufficient in controlling the nonlinearity
in the whole space $\mathbb R^3$.
The beautiful work of Paicu and Zhang \cite{ZHANG1} solved the case when 
the spatial domain is bounded in the $x_1$-direction by observing
a crucial Poincar\'{e} type inequality. Motivated by this Navier-Stokes open problem 
and by experimental observations on the stabilizing effects of
background magnetic fields, this paper intends to understand the global well-posedness 
and stability of a special 3D magnetohydrodynamic (MHD) system near a background magnetic field. 
The spatial domain is $\mathbb R^3$ and the velocity in this MHD system obeys
the 3D Navier-Stokes with only one-directional dissipation. With no Poincar\'{e} type inequality, this problem appears to be impossible. By discovering the mathematical 
mechanism of the experimentally observed stabilizing effect and introducing several innovative techniques to deal with the derivative loss difficulties, we are able to bound the Navier-Stokes nonlinearity and solve the desired global well-posedness and stability problem.
\end{abstract}
\maketitle

\section{Introduction}

\vskip .1in
This paper focuses on a special 3D anisotropic magnetohydrodynamic (MHD) system. The velocity field obeys
the 3D Navier-Stokes equation with one-directional dissipation while the magnetic field satisfies the induction equation  with two-directional magnetic diffusion. More precisely, the MHD system concerned here reads
\begin{equation}\label{mhd0}
	\begin{cases}
		\p_t u + u \cdot \nabla u - \partial_1^2 u  + \nabla P = B \cdot \nabla B, \qquad  x \in \mathbb{R}^3, t > 0,\\
		\p_t B + u \cdot \nabla B - \partial_1^2 B-\partial_2^2 B = B \cdot \nabla u, \\
		\nabla \cdot u = \nabla \cdot B = 0,\\
		u(x,0) =u_0(x), \quad B(x,0) =B_0(x),
	\end{cases}
\end{equation}
where $u=(u_1,u_2,u_3)^T$ represents the velocity field, $P$ the total pressure and $B=(B_1,B_2,B_3)^T$
the magnetic field.  The MHD equations reflect
the basic physics laws governing the motion of electrically conducting fluids such as plasmas, liquid metals
and electrolytes. They are a combination of the Navier-Stokes equation of fluid dynamics
and Maxwell's equation of electromagnetism (see, e.g., \cite{Bis,Davi, Pri}). The MHD system (\ref{mhd0})
focused here is relevant in the modeling of reconnecting plasmas (see, e.g., \cite{CLi1,CLi2}). The Navier-Stokes equation with anisotropic viscous dissipation arise in several physical circumstances.
It can model the turbulent diffusion of rotating fluids in Ekman layers. More details 
on the physical backgrounds
of anisotropic fluids can be found in \cite{CDGG,Ped}.

\vskip .1in
The goal here is to establish the global well-posedness and stability near a background magnetic field. More precisely, the background magnetic field refers to the special
steady-state solution $(u^{(0)}, B^{(0)})$, where
$$
u^{(0)} \equiv 0, \quad B^{(0)} \equiv (0,1, 0) := e_2.
$$
Any perturbation $(u, b)$ near $(u^{(0)}, B^{(0)})$ with
$$
b= B- B^{(0)}
$$
is governed by
\begin{equation}\label{mhd}
	\begin{cases}
		\p_t u + u \cdot \nabla u - \partial_1^2 u + \nabla P = b \cdot \nabla b + \partial_2 b, \quad x \in \mathbb{R}^3, t > 0,\\
		\p_t b + u \cdot \nabla b - \Delta_h b = b \cdot \nabla u + \partial_2 u, \\
		\nabla \cdot u = \nabla \cdot b = 0,\\
		u(x,0) =u_0(x), \quad b(x,0) =b_0(x),
	\end{cases}
\end{equation}
where, for notational convenience, we have written $\Delta_h = \partial_1^2 + \partial_2^2 $
and we shall also write $\na_h =(\p_1, \p_2)$. Our motivation for this study comes from two sources.
The first is to
gain a better understanding on the well-posedness problem on the 3D Navier-Stokes equation with dissipation in only one direction. The second is to reveal and rigorously establish the stabilizing phenomenon exhibited by
electrically conducting fluids. Extensive physical experiments and numerical simulations have been performed
to understand the influence of the magnetic field on the bulk turbulence
involving various electrically conducting fluids such as liquid metals (see, e.g., \cite{AMS,Alex,Alf,BSS,Bur, Davi0,Davi1,Davi,Gall,Gall2}).  These experiments
and simulations have observed a remarkable phenomenon that a background magnetic field can smooth and stabilize
turbulent electrically conducting fluids. We intend to establish these observations as mathematically
rigorous facts on the system (\ref{mhd}).

\vskip .1in
Mathematically the problem we are attempting appears to be impossible. 
The velocity satisfies the 3D incompressible forced Navier-Stokes equations with dissipation in 
only one direction
$$
	\p_t u + u \cdot \nabla u - \partial_1^2 u + \nabla P = b \cdot \nabla b + \partial_2 b, \quad x \in \mathbb{R}^3,\,\, t>0.
$$
However, when the spatial domain is the whole space $\mathbb{R}^3$, the small data global well-posedness of the 3D  Navier-Stokes equations with only one-directional dissipation,
\beq\label{ns}
\begin{cases}
\p_t u + u\cdot\na u = -\na p + \p_1^2 u, \quad x\in \mathbb R^3, \,\,t>0, \\
\na\cdot u=0, \\
u(x,0) =u_0(x)
\end{cases}
\eeq
remains an outstanding open problem. The difficulty is immediate. 
The dissipation in one direction is simply not sufficient in controlling the nonlinearity when the 
spatial domain is the whole space $\mathbb R^3$. 

\vskip .1in 
In a beautiful work \cite{ZHANG1}, Paicu and Zhang  were able to deal with the case when 
the spatial domain is bounded in the $x_1$-direction with Dirichlet type boundary conditions. 
They successfully established the small data global well-posedness by observing a crucial 
Poincar\'{e} type inequality. This inequality allows one to bound $u$ in terms of $\p_1 u$ and thus leads
to the control of the nonlinearity. However, such Poincar\'{e} type inequalities are not valid for the whole space case. 

\vskip .1in
If we increase the dissipation to be in two-directions, say
$$
\begin{cases}
	\p_t u + u\cdot\na u = -\na p + (\p_1^2+ \p_2^2) u, \quad x\in \mathbb R^3, \,\,t>0, \\
	\na\cdot u=0,
\end{cases}
$$
then any sufficiently small initial data in a suitable Sobolev or Besov space always leads to a
global (not necessarily stable) solution. There are substantial developments 
on the 3D anisotropic Navier-Stokes equations with two-directional dissipation. 
Significant progress has been made on the global existence of
small solutions and on the regularity criteria on general large solutions in various Sobolev and Besov settings (see, e.g., \cite{CDGG,If,ZHANG2,ZHANG1, Pai2}).

\vskip .1in
We return to the well-posedness and stability problem proposed here. Clearly, if
the coupling with the magnetic field does not generate extra smoothing and stabilizing effect, 
then the problem focused here would be impossible. Fortunately, we discover in this paper that 
the magnetic field does help stabilize the conducting fluids, as observed by physical experiments and numerical simulations.
To unearth the stabilization effect, we take advantage of the coupling and intersection in the MHD system
to convert (\ref{mhd}) into the following wave equations
\begin{equation}\label{wav}
	\begin{cases}
\p_t^2 u - (\p_1^2 + \Delta_h) \p_t u + \p_1^2 \Delta_h u- \p_2^2 u= (\p_t -\Delta_h) N_1 + \p_2 N_2,\\
\p_t^2 b - (\p_1^2 + \Delta_h) \p_t b + \p_1^2 \Delta_h b- \p_2^2 b= (\p_t -\p_1^2) N_2 + \p_2 N_1,
	\end{cases}
\end{equation}
where $N_1$ and $N_2$ are the nonlinear terms
$$
N_1 = \mathbb P (-u\cdot\na u+ b\cdot \na b), \quad N_2 = -u\cdot\na b + b\cdot\na u
$$
with $\mathbb P=I -\na\Delta^{-1} \na\cdot$ being the Leray projection. (\ref{wav}) is derived by taking the
time derivative of (\ref{mhd}) and making several substitutions. In comparison with (\ref{mhd}), the wave structure in (\ref{wav}) exhibits much more regularity properties. In particular, the linearized wave equation of $u$,
$$
\p_t^2 u - (\p_1^2 + \Delta_h) \p_t u + \p_1^2 \Delta_h u- \p_2^2 u=0
$$
reveals the regularization of $u$ in the $x_2$-direction, although the original system (\ref{mhd}) involves only
the dissipation in the $x_1$-direction.

\vskip .1in
Unfortunately the extra regularization in the $x_2$-direction is not sufficient to control
the Navier-Stokes nonlinearity. The regularity from the wave structure is in general
one-derivative-order lower
than the standard dissipation. More precisely, when we seek solutions in the Sobolev
space $H^4(\mathbb R^3)$, the dissipation in the $x_1$-direction yields the time integrability term
\beq\label{jj}
\int_0^t \|\p_1 u(\tau)\|^2_{H^4}\,d\tau,
\eeq
but the extra regularity in the $x_2$-direction due to the background magnetic field and the coupling can only allow us to bound
\beq\label{jj1}
\int_0^t \|\p_2 u(\tau)\|^2_{H^3}\,d\tau.
\eeq
But (\ref{jj}) and (\ref{jj1}) may not be sufficient to control some of the terms resulting from the nonlinearity
in the estimate of $\|u\|_{H^4}$ such as
$$
\int_{\mathbb R^3} \p_3 u_3 \,\p_3^4 u_1\, \p_3^4 u_1\,dx.
$$
Naturally we use the divergence-free condition $\p_3 u_3 = -\p_1 u_1 -\p_2 u_2$ to eliminate the bad derivative
$\p_3$, but this process generates a new difficult term
$$
\int_{\mathbb R^3} \p_2 u_2 \,\p_3^4 u_1\, \p_3^4 u_1\,dx,
$$
which can not be bounded in terms of (\ref{jj1}). If we integrate by parts, we would have the fifth-order
derivatives on the velocity, which can not be controlled. We call this phenomenon the derivative loss
problem.   The above analysis reveals that the extra regularity
gained through the background magnetic field and the nonlinear coupling is not sufficient to deal with the
derivative loss problem.

\vskip .1in
This paper creates several new techniques to combat the derivative loss problem. As a consequence, we are
able to offer suitable upper bounds on the Navier-Stokes nonlinearity and solve the desired global well-posedness
and stability problem.  We state our main result and then  describe these techniques.

\begin{theorem} \label{main}
	Consider \eqref{mhd} with the initial datum $(u_0, b_0) \in H^4(\mathbb{R}^3)$
	and $\nabla \cdot u_0 = \nabla \cdot b_0 = 0$, Then there exists a constant $\epsilon>0$ such that, if
	\begin{equation*}\label{initialdata}
		\|u_0\|_{H^4} + \|b_0\|_{H^4} \leq \epsilon,
	\end{equation*}
	system \eqref{mhd} has a unique global classical solution $(u, b)$ satisfying, for any $t>0$,
	$$
	\|u(t)\|_{H^4}^2 + \|b(t)\|_{H^4}^2 + \int_0^t \left(\|\partial_1 u\|_{H^4}^2 + \|\na_h b\|^2_{H^4} + \|\p_2 u\|^2_{H^3}\right) \,d\tau \leq \;\epsilon.
	$$	
\end{theorem}

We make two remarks. Theorem \ref{main} does not solve the small data global well-posedness problem on the 3D Navier-Stokes equations in (\ref{ns}). The MHD system in (\ref{mhd}) with $b\equiv 0$ does not reduce to (\ref{ns}). It is hoped that
the new idea and techniques discovered in this paper will shed light on the open problem on (\ref{ns}).

\vskip .1in
A previous work of Wu and Zhu \cite{WuZhu} successfully resolved the small data global well-posedness and stability problem on the 3D MHD system with horizontal dissipation and vertical magnetic diffusion near an equilibrium,
\begin{equation}\label{mhd1}
	\begin{cases}
		\p_t u + u \cdot \nabla u - \Delta_h u + \nabla P = b \cdot \nabla b + \partial_1 b, \quad  x \in \mathbb{R}^3, t > 0,\\
		\p_t b + u \cdot \nabla b - \p_3^2 b = b \cdot \nabla u + \partial_1 u, \\
		\nabla \cdot u = \nabla \cdot b = 0.
	\end{cases}
\end{equation}
Although the small data global well-posedness problem on (\ref{mhd1}) is highly nontrivial, it is clear
that the current problem on (\ref{mhd}) is different and can be even more difficult. One simple reason is that
the equation of $u$ in (\ref{mhd1}) contains dissipation in two directions and all the nonlinear terms in the
estimate of the Sobolev norms involve $u$ as a component. However, (\ref{mhd}) has only one-direction velocity
dissipation and the velocity nonlinear term involves only $u$ (no other more regularized components).
The methods in \cite{WuZhu} are not sufficient for the problem of this paper.

\vskip .1in
Since the pioneering work of F. Lin and P. Zhang \cite{LinZh}
and F. Lin, L. Xu and P. Zhnag \cite{LinZhang1}, many efforts have now devoted to the small data global
well-posedness and stability problems on partially dissipated MHD systems. MHD systems with various levels of
dissipation and magnetic diffusion near several steady-states have been thoroughly investigated and a rich array
of results have been established (see, e.g., \cite{AZ,BSS, CaiLei, HeXuYu,HuWang, JiangARMA, JiangJMPA, LinZhang1,LinZh, PanZhouZhu,Ren, Ren2, Tan, WeiZ, WuWu, WuWuXu, ZT1, ZT2,ZhouZhu}). In addition,
global well-posedness on the MHD equations with general large initial data has also been actively pursued
and important progress has been made (see, e.g., {\cite{CaoWuYuan,CMZ,DongLiWu1, DongLiWu0, FNZ,
	Fefferman1, Fefferman2, HuangLi, JNW, JiuZhao2, LinDu, WuMHD2018, Yam1, YuanZhao, Ye})}. Needless to say,
the references listed here represent only a small portion of the very large literature
on the global well-posedness and related problems concerning the MHD equations.

\vskip .2in
We briefly outline the proof of Theorem \ref{main}. We have chosen the Sobolev space $H^4$ as the functional
setting for our solutions since $H^3$ does not appear to be regular enough to accommodate our approach. Since
the local well-posedness of (\ref{mhd}) in $H^4$ follows from a standard procedure (see, e.g., \cite{MaBe}),
the proof focuses on the global $H^4$-bound. The framework of the proof for the global $H^4$-bound is the
bootstrapping argument (see, e.g., \cite[p.20]{Tao}).  The process starts with the setup of a suitable energy functional. Besides the
standard $H^4$-energy, we also need to include the extra regularization term in (\ref{jj1}) resulting from
the background magnetic field and the coupling. More precisely, we set
\beq\label{ed}
\mathscr{E}(t) = \mathscr{E}_1(t) +  \mathscr{E}_2(t),
\eeq
where
\begin{align*}
& \mathscr{E}_1 (t)=\sup_{0 \leq \tau \leq t} \big( \|u\|_{H^4}^2 + \|b\|_{H^4}^2 \big)+\int_0^t \Big( \|\p_1u\|_{H^4}^2 + \|\na_h b\|_{H^4}^2\Big) d\tau,\\
& \mathscr{E}_2(t) =\int_0^t \|\p_2 u\|_{H^3}^2  d\tau.
\end{align*}
Our main efforts are devoted to showing that, for some constant $C_0>0$ and for all $t>0$,
\begin{equation}\label{bb}
	\mathscr{E} (t) \leq C_0\, \mathscr{E} (0)+ C_0\, \mathscr{E}^{\frac32} (0) + C_0\,\mathscr{E}^{\frac{3}{2}} (t) + C_0\, \mathscr{E}^2 (t).
\end{equation}
Then an application of the bootstrapping argument to (\ref{bb}) would yield the desired result, namely, for
a sufficiently small $\epsilon>0$,
$$
\|u_0\|_{H^4} + \|b_0\|_{H^4} \le \epsilon \quad \mbox{or}\quad \mathscr{E}(0) \le \epsilon^2
$$
would imply, for a constant $C>0$ and for any $t>0$,
$$
\mathscr{E}(t) \le C\,\epsilon^2.
$$
It then yields the global uniform $H^4$-bound on $(u, b)$ and the stability.

\vskip .1in
The proof of (\ref{bb}) is highly nontrivial. It is achieved by establishing the following two inequalities
for positive constants $C_1$ through $C_8$,
\begin{align}
	&{\mathscr{E}_{2}(t)}\le C_1\, \mathscr{E}(0) + C_2\,\mathscr{E}_1(t) +  C_3\,\mathscr{E}_1^{\frac32}(t)
	+ C_4 \,\mathscr{E}_2^{\frac32}(t), \label{lb}\\
	&{\mathscr{E}_{1} (t)}\le C_5\, \mathscr{E}(0) + C_6\, \mathscr{E}^{\frac32} (0) + C_7\, \mathscr{E}^{\frac32}(t) + C_8\, \mathscr{E}^2 (t), \label{hb}
\end{align}
which clearly lead to (\ref{bb}) by adding (\ref{hb}) to a suitable multiple of (\ref{lb}). Due to the
equivalence of the norms
$$
\|f\|_{H^k}  {\; \sim \; }  \|f\|_{L^2} + \sum_{i=1}^3 \|\p_i^k f\|_{L^2},
$$
the verification of (\ref{lb}) is naturally divided into the estimates of
$$
\int_0^t \|\p_2 u\|_{L^2}^2 \,d\tau \quad \mbox{and} \quad
\sum_{i=1}^3\int_0^t \|\p_i^3 \p_2 u\|_{L^2}^2 \,d\tau.
$$
One key strategy is to take advanatge of the coupling and interaction of the MHD system in (\ref{mhd})
to shift the time integrability. More precisely, we replace $\p_2 u$ by the evolution of $b$,
$$
\p_2 u = \p_t b +u\cdot\nabla b-\Delta_h b-b\cdot\nabla u
$$
and convert the time integral of $\|\p_2 u\|_{L^2}^2$ into time integrals of more regular terms,
$$
\int_0^t \|\p_2 u\|_{L^2}^2 \,d\tau = \int_0^t \int_{\mathbb R^3} (\p_t b +u\cdot\nabla b-\Delta_h b-b\cdot\nabla u)\cdot \p_2u \;dx\,d\tau.
$$
We then further shift the time derivative from $b$ to $\p_2 u$ and involve the equation of $u$.
This process generates many more terms, but it replaces those with less time-integrable terms by more
integrable nonlinear terms. More details can be found in Section \ref{e2b}.

\vskip .1in
It is much more difficult to verify (\ref{hb}). Undoubtedly the most difficult term is generated by the
Navier-Stokes nonlinear term. We use the vorticity formulation to take advantage of certain symmetries. Since
$\|\om\|_{L^2} = \|\na u||_{L^2}$ for the vorticity $\om =\na\times u$, it suffices to  control
$\|\om\|_{\dot H^3}$. One of the wildest terms is given by
$$
\int \p_3 u_3\, \p_3^3 \om \cdot \p_3^3 \om \,dx .
$$
Naturally we eliminate the bad derivative $\p_3$ via the divergence-free condition $\p_3 u_3 = -\p_1 u_1 -\p_2 u_2$,
but this leads to another uncontrolled term
\beq\label{aa}
\int \p_2 u_2\, \p_3^3 \om \cdot \p_3^3 \om \,dx.
\eeq
As aforementioned, $\mathscr{E}_{2}$ contains only fourth-order derivatives and integrating by parts in
(\ref{aa}) would generate uncontrollable fifth-order derivatives. Obtaining a suitable bound on
(\ref{aa}) appears to be an impossible mission. This leads to the derivative loss problem.
This paper introduces several new techniques to unearth the hidden
structure in the nonlinearity. We enhance the nonlinearity through the coupling in the system and induce
cancellations through the construction of artificial symmetries. More precisely, we are able to establish
the desired upper bounds stated in the following proposition. For notational convenience, we use $A \lesssim B$
to mean $A\le C\, B$ for a pure constant $C>0$.
\begin{proposition}\label{lem-important}
	Let $(u, b)\in H^4$ be a solution of (\ref{mhd}). Let $\omega = \nabla \times u$ and $H = \nabla \times b$
	be the corresponding vorticity and current density, respectively. Let $\mathscr{E}(t)$ be defined as in  \eqref{ed}. Let $\mathcal{W}(t)$ be the interaction type terms  defined as follows,
	$$ \mathcal{W}^{ijk}(t) = \int_{\mathbb{R}^3} \p_3^{3}\omega_i \p_2u_j \p_3^{3}\omega_k \; dx \quad \text{for} \; \; i,j,k \in \{1,2,3\}.$$
	Then the time integral of $\mathcal{W}^{ijk}$ admits the following bound,
	$$ \int_0^t \mathcal{W}^{ijk} (\tau) \; d\tau  \lesssim \mathscr{E}^\frac{3}{2}(0) +\mathscr{E}^{\frac{3}{2}}(t) + \mathscr{E}^2 (t). $$
\end{proposition}
The proof of this proposition is not trivial. When we just have the 3D Navier-Stokes with one-directional dissipation,
this term can not be suitably bounded and the small-data global well-posedness remains an open problem for the 3D Navier-Stokes.  The advantage of working with the MHD system in (\ref{mhd}) is the coupling and interaction. We take advantage of this coupling
to replace $\p_2 u_j$ via the equation of $b$,
\begin{equation*}\label{W}
	\begin{split}
		\mathcal{W}^{ijk}(t)=\; & \int_{\mathbb{R}^3} \p_3^{3}\omega_i \Big[\p_t b_j +u\cdot\nabla b_j -\Delta_h b_j -b\cdot\nabla u_j\Big] \p_3^{3}\omega_k \; dx \\
		= \; & \frac{d}{dt} \int_{\mathbb{R}^3} \p_3^{3}\omega_i b_j \p_3^{3}\omega_k dx- \int_{\mathbb{R}^3} b_j \p_t(\p_3^{3}\omega_i \p_3^{3}\omega_k)dx \\
		&+ \int_{\mathbb{R}^3} \p_3^{3}\omega_i u\cdot\nabla b_j \p_3^{3}\omega_kdx - \int_{\mathbb{R}^3} \p_3^{3}\omega_i b\cdot\nabla u_j \p_3^{3}\omega_kdx \\
		& -\int_{\mathbb{R}^3} \p_3^{3}\omega_i \Delta_h b_j \p_3^{3}\omega_kdx.
	\end{split}
\end{equation*}
Immediately we encounter the new difficult term
$$
-\int_{\mathbb{R}^3} b_j \p_t(\p_3^{3}\omega_i \p_3^{3}\omega_k)dx,
$$
which is further converted into integrals of many more terms after invoking the evolution of the vorticity,
\begin{equation}\nonumber
	\begin{split}
		&-\int_{\mathbb{R}^3} b_j \p_t(\p_3^{3}\omega_i \p_3^{3}\omega_k)dx\\
 &=\int_{\mathbb{R}^3} b_j \p_3^{3}\omega_i \p_3^{3}\big( u\cdot\nabla\omega_k - \omega\cdot\nabla u_k -\p_1^2 \omega_k -b\cdot\nabla H_k +H \cdot\nabla b_k -\p_2 H_k \big) \\
\;& \qquad + b_j \,\p_3^{3}\omega_k \p_3^{3}\big( u\cdot\nabla\omega_i - \omega\cdot\nabla u_i -\p_1^2 \omega_i -b\cdot\nabla H_i +H \cdot\nabla b_i -\p_2 H_i \big)\,dx.
	\end{split}
\end{equation}
Some of the terms above can be paired together to form symmetries to deal with the derivative loss problem. This
process also allows us to convert some of the cubic nonlinearity into quartic nonlinearity, which helps
improve the time integrability. However, there are terms that can not be paired into symmetric structure, for example,
$$
\int_{\mathbb{R}^3} b_j \big[ -\p_3^{3}\omega_i b\cdot\nabla \p_3^{3}H_k  - \p_3^{3}\omega_k b\cdot\nabla\p_3^{3} H_i\big]\; dx.
$$
Our idea in dealing with such terms is to construct artificial symmetries by adding and subtracting suitable terms.
This strategy helps us alleviate the derivative loss problem eventually. The technical details are very complicated
and are left to the proof of Proposition \ref{lem-important} in  Section \ref{proof-lem}.

\vskip .1in
The rest of this paper is divided into four sections. Section \ref{e2b} proves (\ref{lb}), one of the two
key energy inequalities while Section \ref{proof-high} establishes the second key energy inequality, namely
(\ref{hb}). Section \ref{proof-lem} provides the proof of Proposition \ref{lem-important} and deals some of the
most difficult terms in the Navier-Stokes nonlinearity in Lemma \ref{lem-assistant}.
The last section, Section \ref{proof-thm}, finishes the proof of Theorem \ref{main}.

\vskip .3in
\section{Proof of (\ref{lb})}
\label{e2b}

This section proves (\ref{lb}), namely, for four positive constants $C_1$ through $C_4$,
\beq\label{bb1}
\mathscr{E}_{2}\le C_1\, \mathscr{E}(0) + C_2\,\mathscr{E}_1(t) +  C_3\,\mathscr{E}_1^{\frac32}(t)
+ C_4 \,\mathscr{E}_2^{\frac32}(t) \quad\mbox{for all $t>0$.}
\eeq
 The following lemma provides a powerful tool to control the triple products
in terms of anisotropic upper bounds.

\begin{lemma}\label{lem-anisotropic-est}
	The following inequalities hold when the right-hand sides are all bounded,
	\begin{equation}\nonumber
		\begin{split}
			\int_{\mathbb{R}^3} |f g h| \;dx \lesssim \;& \|f\|_{L^2}^{\frac{1}{2}} \|\partial_1 f\|_{L^2}^{\frac{1}{2}}\|g\|_{L^2}^{\frac{1}{2}} \|\partial_2 g\|_{L^2}^{\frac{1}{2}}\|h\|_{L^2}^{\frac{1}{2}} \|\partial_3 h\|_{L^2}^{\frac{1}{2}},\\
			 {\int_{\mathbb{R}^3} |f g h| \;dx \lesssim }\;& {\|f\|_{L^2}^{\frac{1}{4}} \|\partial_i f\|_{L^2}^{\frac{1}{4}}
				\|\partial_j f\|_{L^2}^{\frac{1}{4}}\|\partial_i\p_j f\|_{L^2}^{\frac{1}{4}}\|g\|_{L^2}^{\frac{1}{2}} \|\partial_k g\|_{L^2}^{\frac{1}{2}}\|h\|_{L^2}}\\
			\lesssim \;&  {\|f\|_{H^1}^{\frac{1}{2}} \|\partial_i f\|_{H^1}^{\frac{1}{2}}\|g\|_{L^2}^{\frac{1}{2}}
				\|\partial_k g\|_{L^2}^{\frac{1}{2}}\|h\|_{L^2},}\\
			\left(\int_{\mathbb{R}^3} |f g|^2 \;dx\right)^\frac{1}{2} \lesssim \;&\|f\|_{L^2}^{\frac{1}{4}} \|\partial_i f\|_{L^2}^{\frac{1}{4}} \|\partial_j f\|_{L^2}^{\frac{1}{4}} \|\partial_i \partial_j f\|_{L^2}^{\frac{1}{4}} \|g\|_{L^2}^{\frac{1}{2}} \|\partial_k g\|_{L^2}^{\frac{1}{2}}\\
			\lesssim \;&\|f\|_{H^1}^{\frac{1}{2}} \|\partial_i f\|_{H^1}^{\frac{1}{2}}  \|g\|_{L^2}^{\frac{1}{2}} \|\partial_k g\|_{L^2}^{\frac{1}{2}}, \\
			\int_{\mathbb{R}^3} |f g h v| \;dx \lesssim \;& \|f\|_{L^2}^{\frac{1}{4}} \|\partial_i f\|_{L^2}^{\frac{1}{4}} \|\partial_j f\|_{L^2}^{\frac{1}{4}} \|\partial_i \partial_j f\|_{L^2}^{\frac{1}{4}}\\
			&  \cdot\|g\|_{L^2}^{\frac{1}{4}} \|\partial_i g\|_{L^2}^{\frac{1}{4}} \|\partial_j g\|_{L^2}^{\frac{1}{4}} \|\partial_i \partial_i g\|_{L^2}^{\frac{1}{4}} \\
			& \cdot  \|h\|_{L^2}^{\frac{1}{2}} \|\partial_k h\|_{L^2}^{\frac{1}{2}} \|v\|_{L^2}^{\frac{1}{2}} \|\partial_k v\|_{L^2}^{\frac{1}{2}} \\
			\lesssim  \;&  \|f\|_{H^1}^{\frac{1}{2}} \|\partial_i f\|_{H^1}^{\frac{1}{2}} \|g\|_{H^1}^{\frac{1}{2}} \|\partial_i g\|_{H^1}^{\frac{1}{2}}  \|h\|_{L^2}^{\frac{1}{2}} \|\partial_k h\|_{L^2}^{\frac{1}{2}} \|v\|_{L^2}^{\frac{1}{2}} \|\partial_k v\|_{L^2}^{\frac{1}{2}}, \\
		\end{split}
	\end{equation}
	where $i,j$ and $k$ belong to $\{1,2,3\}$ are different numbers.
\end{lemma}

The proof of Lemma \ref{lem-anisotropic-est} relies on the following one-dimensional interpolation inequality,
for $f \in H^1(\mathbb R)$,
\begin{equation}\nonumber
	\|f\|_{L^\infty (\mathbb{R})} \leq \sqrt{2} \|f\|_{L^2(\mathbb{R})}^\frac{1}{2} \|f'\|_{L^2(\mathbb{R})}^\frac{1}{2}.
\end{equation}
A detailed proof of this lemma can be found in \cite{WuZhu}. We remark that similar anisotropic inequalities
are also available for two-dimensional functions (see \cite{CaoWu}).

\vskip .1in
We are now ready to prove (\ref{bb1}).
\begin{proof}[Proof of (\ref{bb1})] Due to the norm equivalence, namely for any integer $k>0$,
	\beq\label{ne}
	\|f\|_{H^k}^2 { \;\sim \;}  \|f\|_{L^2}^2 + \sum_{i=1}^3 \|\p_i^k f\|_{L^2}^2,
	\eeq
	 it suffices to bound
	 $$
	 \int_0^t \|\p_2u\|_{L^2}^2\,d\tau \quad \mbox{and}\quad  	\sum_{i=1}^3  \int_0^t \| \p_i^{3} \p_2 u \|_{L^2}^2\,d\tau.
	 $$
	Recalling the equations in (\ref{mhd}),
	\begin{equation}\label{er}
		\begin{split}
		&\p_2 u = \p_t b+u\cdot\nabla b-\Delta_h b-b\cdot\nabla u, \\
		&\p_t u = {-u\cdot\nabla u + \p_1^2 u -\na P +b\cdot\nabla b +\p_2b ,}
\end{split}
\end{equation}
we have
\begin{equation}\label{4eq2}
\begin{split}
  \|\p_2u\|_{L^2}^2 =& \int_{\mathbb{R}^3} (\p_t b+u\cdot\nabla b-\Delta_h b-b\cdot\nabla u)\cdot \p_2u \;dx \\
  =& \; \frac{d}{dt} \int_{\mathbb{R}^3} b \cdot \p_2u \;dx +\int_{\mathbb{R}^3}\p_2b \cdot \p_t u dx\\
  &+\int_{\mathbb{R}^3}(u\cdot\nabla b-\Delta_h b-b\cdot\nabla u)\cdot \p_2u \;dx\\
   =& \; \frac{d}{dt} \int_{\mathbb{R}^3} b \cdot \p_2u \;dx \\
   & + \int_{\mathbb{R}^3} \p_2b \cdot (\p_2b+b\cdot\nabla b-\nabla P + \p_1^2u-u\cdot\nabla u)\;dx\\
   &+\int_{\mathbb{R}^3}(u\cdot\nabla b-\Delta_h b-b\cdot\nabla u)\cdot \p_2u \;dx.
\end{split}
\end{equation}
Due to $\na\cdot b=0$,
$$
\int_{\mathbb{R}^3} \p_2b \cdot \na P\,dx =0.
$$
We bound the nonlinear terms. By  Sobolev's inequality and Lemma \ref{lem-anisotropic-est},
\begin{equation}\nonumber
\begin{split}
&\int_{\mathbb{R}^3} \p_2b \cdot (b\cdot\nabla b)\;dx =\int_{\mathbb{R}^3}\big( \p_2b \cdot (b_h \cdot \na_hb) + \p_2b \cdot (b_3\p_3b) \big)\;dx \\
& \lesssim  \|\na_hb\|_{L^2}^2\|b\|_{H^2} + \|\p_2b\|_{L^2}^\frac{1}{2}\|{\bf \p_3}\p_2 b\|_{L^2}^\frac{1}{2} \|b_3\|_{L^2}^\frac{1}{2}\|{\bf \p_1} b_3\|_{L^2}^\frac{1}{2}
\|\p_3b\|_{L^2}^\frac{1}{2}\|{\bf \p_2} \p_3b\|_{L^2}^\frac{1}{2} \\
& \lesssim \|b\|_{H^2} \|\na_h b\|_{H^1}^2,
\end{split}
\end{equation}
\begin{equation}\nonumber
\begin{split}
&\int_{\mathbb{R}^3} \p_2b \cdot (u\cdot\nabla u) \;dx =\int_{\mathbb{R}^3}{\big(} \p_2b \cdot (u_h \cdot \na_hu)+\p_2b \cdot (u_3\p_3u) {\big)}\;dx \\
& \lesssim  \|\p_2b\|_{L^2}\|\na_hu\|_{L^2}\|u\|_{H^2} + \|\p_2 b\|_{L^2}^\frac{1}{2} \| {\bf \p_3} \p_2 b\|_{L^2}^\frac{1}{2} \|u_3\|_{L^2}^\frac{1}{2}\|{\bf \p_2} u_3\|_{L^2}^\frac{1}{2}
\|\p_3 u\|_{L^2}^\frac{1}{2} \|{\bf \p_1}\p_3 u\|_{L^2}^\frac{1}{2} \\
& \lesssim \|\p_2b\|_{L^2}\|\na_hu\|_{L^2}\|u\|_{H^2} + \|\p_2 b\|_{H^1} \|\p_2 u\|_{L^2}^\frac{1}{2} \|\p_1 u\|_{H^1}^\frac{1}{2}\|u\|_{H^1},
\end{split}
\end{equation}
\begin{equation}\nonumber
	\begin{split}
		& \int_{\mathbb{R}^3}  (u\cdot\nabla b) \cdot \p_2u\;dx \\
		= & \int_{\mathbb{R}^3} {\big(} (u_h\cdot \na_h b) \cdot \p_2u+ (u_3\p_3b) \cdot \p_2u { \big)} \;dx \\
		\lesssim \; & \|\na_hb\|_{L^2}\|\p_2u\|_{L^2}\|u\|_{H^2}
		+ \|u_3\|_{L^2}^\frac{1}{2} \|{\bf \p_1} u_3\|_{L^2}^\frac{1}{2} \|\p_3 b\|_{L^2}^\frac{1}{2} \|{\bf \p_2} \p_3 b\|_{L^2}^\frac{1}{2} \|\p_2 u\|_{L^2}^\frac{1}{2} \|\p_2 {\bf \p_3} u\|_{L^2}^\frac{1}{2}
	\end{split}
\end{equation}
and
\begin{equation}\nonumber
	\begin{split}
		& \int_{\mathbb{R}^3}  (b\cdot\nabla u) \cdot \p_2u \;dx \\
		= & \int_{\mathbb{R}^3} (b_h\cdot \na_h u) \cdot \p_2u+ (b_3\p_3u) \cdot \p_2u dx \\
		\lesssim \; & \|b\|_{H^2} \|\na_h u\|_{L^2}^2 + \|b_3\|_{L^2}^\frac{1}{2} \|{\bf \p_2} b_3 \|_{L^2}^\frac{1}{2} \|\p_3 u\|_{L^2}^\frac{1}{2}\|{\bf \p_1} \p_3 u\|_{L^2}^\frac{1}{2}
		\|\p_2 u\|_{L^2}^\frac{1}{2} \|\p_2 {\bf \p_3} u\|_{L^2}^\frac{1}{2}.
	\end{split}
\end{equation}
Inserting these bounds in (\ref{4eq2}), integrating in time and using the simple bound
$$
\int_0^t {\frac{d}{d\tau}} \int_{\mathbb{R}^3} b \cdot \p_2u \;dx \; {d\tau} \le \|b(t)\|_{L^2} \|{\p_2 u(t)}\|_{L^2}
+ \|b_0\|_{L^2} \|\p_2 u_0 \|_{L^2} \le \mathscr{E}_{1}(t) + \mathscr{E}_{1}(0),
$$
we obtain
\begin{equation}\label{4eq4}
	\begin{split}
		\int_0^t \|\p_2 u\|_{L^2}^2 \; d\tau \lesssim \; & \mathscr{E}_{1}(0) +\mathscr{E}_{1}(t)+
		+ \mathscr{E}_{1}^\frac{3}{2}(t) + \mathscr{E}_{1}(t)\mathscr{E}_{2}^\frac{1}{2}(t)\\
		& + \mathscr{E}_{1}^\frac{1}{2}(t)\mathscr{E}_{2}(t)+\mathscr{E}_{1}^\frac{5}{4}(t)\mathscr{E}_{2}^{^\frac{1}{4}}(t)\\
		\lesssim \; & \mathscr{E}_{1}(0) +\mathscr{E}_{1}(t)+  \mathscr{E}_{1}^\frac{3}{2}(t) +\mathscr{E}_{2}^\frac{3}{2}(t){.}
	\end{split}
\end{equation}
Here we have used several H\"{o}lder's inequalities such as
\begin{align*}
	& \sup_{0 \leq \tau \leq t} \|b(\tau)\|_{H^2} \int_0^t \|\na_h b\|_{H^1}^2 \; d\tau \le \mathscr{E}_{1}^\frac{3}{2}(t), \\
	& \sup_{0 \leq \tau \leq t} \|u(\tau)\|_{H^2} \int_0^t \|\p_2 b\|_{L^2} \|\p_2 u\|_{L^2} \; d\tau \le  \mathscr{E}_{1}(t)\mathscr{E}_{2}^\frac{1}{2}(t) \le  \mathscr{E}_{1}^\frac{3}{2}(t) +\mathscr{E}_{2}^\frac{3}{2}(t).
\end{align*}
We now turn to the bound for the highest-order derivatives. By \eqref{er},
\begin{equation}\label{5eq6}
	\begin{split}
		\sum_{i=1}^3 \| \p_i^{3} \p_2 u \|_{L^2}^2
		= & \sum_{i=1}^3 \int_{\mathbb{R}^3} \p_i^{3} \Big(\p_t b +u\cdot\nabla b -\Delta_hb -b\cdot\nabla u \Big) \cdot  \p_i^{3}\p_2u dx\\
		=& \sum_{i=1}^3  \frac{d}{dt} \int_{\mathbb{R}^3} \p_i^{3}b \cdot \p_i^{3}\p_2 u dx + \sum_{i=1}^3 \int_{\mathbb{R}^3} \p_i^{3}\p_2 b \cdot \p_i^{3}\p_tu\, dx \\
		&  + \sum_{i=1}^3 \int_{\mathbb{R}^3} \p_i^{3} \Big( u\cdot\nabla b -\Delta_hb -b\cdot\nabla u \Big)\cdot \p_i^{3}\p_2u dx.
	\end{split}
\end{equation}
The estimates for the terms with $i = 1, 2$ (those containing $\p_1^{3}$ or $\p_2^{3}$ derivatives) are simple.
We focus on the terms with $i=3$ (those with the bad derivative $\p_3$), namely,
\begin{equation}\label{nn}
	\begin{split}
& \frac{d}{dt} \int_{\mathbb{R}^3} \p_3^{3}b \cdot \p_3^{3}\p_2 u dx + \int_{\mathbb{R}^3} \p_3^{3}\p_2 b \cdot \p_3^{3}u_t dx \\
& + \int_{\mathbb{R}^3} \p_3^{3} \Big( u\cdot\nabla b -\Delta_hb -b\cdot\nabla u \Big)\cdot \p_3^{3}\p_2u dx.
	\end{split}
\end{equation}
The second part of (\ref{nn}) can be handled as follows. By Lemma \ref{lem-anisotropic-est},
\begin{equation}\nonumber
	\begin{split}
		&\int_{\mathbb{R}^3} \p_3^{3}\p_2 b \cdot \p_3^{3} \p_t u\, dx \\
		= &\int_{\mathbb{R}^3} \p_3^{3}\p_2 b \cdot \p_3^{3}(\p_2b + b\cdot\nabla b -\nabla P +\p_1^2u-u\cdot\nabla u) dx \\
		\lesssim & \; \|\p_2 b\|_{H^{3}}{\big(}\|\p_2 b\|_{H^{3}} + \|\p_1 u\|_{H^4}{\big)} \\
		& + \|\p_2 b\|_{H^{3}}^\frac{1}{2} \|{\bf \p_3} \p_2 b\|_{H^{3}}^\frac{1}{2} {\big(}\|b\|_{H^{3}}^\frac{1}{2}\|{\bf \p_2} b\|_{H^{3}}^\frac{1}{2}
		\|b\|_{H^4}^\frac{1}{2}\|{\bf \p_1} b\|_{H^4}^\frac{1}{2} \\
		& + \|u\|_{H^{3}}^\frac{1}{2}\|{\bf \p_2} u\|_{H^{3}}^\frac{1}{2}
		\|u\|_{H^4}^\frac{1}{2}\|{\bf \p_1} u\|_{H^4}^\frac{1}{2}{\big)}.
	\end{split}
\end{equation}
The last part in (\ref{nn}) can be bounded by
\begin{equation}\nonumber
	\begin{split}
		& \int_{\mathbb{R}^3} \p_3^{3} \Big( u\cdot\nabla b -\Delta_hb -b\cdot\nabla u \Big) \cdot \p_3^{3}\p_2u dx \\
		\lesssim & \; \|\p_2 u\|_{H^{3}} \big( \|\na_h b\|_{H^4} + \|u\|_{H^{3}}^\frac{1}{4} \|{\bf \p_1} u\|_{H^{3}}^\frac{1}{4} \|{\bf \p_3} u\|_{H^{3}}^\frac{1}{4}\|{\bf \p_1 \p_3} u\|_{H^{3}}^\frac{1}{4} \|b\|_{H^4}^\frac{1}{2}\|{\bf \p_2} b\|_{H^4}^\frac{1}{2} \\
		& + \|b\|_{H^{3}}^\frac{1}{4} \|{\bf \p_2} b\|_{H^{3}}^\frac{1}{4} \|{\bf \p_3} b\|_{H^{3}}^\frac{1}{4}\|{\bf \p_2 \p_3} b\|_{H^{3}}^\frac{1}{4} \|u\|_{H^4}^\frac{1}{2}\|{\bf \p_1} u\|_{H^4}^\frac{1}{2}\big).
	\end{split}
\end{equation}
The terms with $i=1,2$ in \eqref{5eq6} are simpler and can be bounded similarly by applying
Lemma \ref{lem-anisotropic-est}. Inserting the bounds above in \eqref{5eq6} and integrating in time yields
\begin{equation}\label{5eq7}
	{\sum_{i=1}^3\int_0^t \|\p_i^3\p_2 u\|_{L^2}^2 }d\tau \lesssim \mathscr{E}_{1}(0) +\mathscr{E}_{1}(t)+  \mathscr{E}_{1}^\frac{3}{2}(t) +\mathscr{E}_{2}^\frac{3}{2}(t).
\end{equation}
Combining (\ref{4eq4}) and (\ref{5eq7}) gives (\ref{bb1}).  This finishes the proof for \eqref{lb}.
\end{proof}

\vskip .3in
\section{Proof of (\ref{hb})}
\label{proof-high}

\vskip .1in
This section proves the energy inequality in (\ref{hb}).

\begin{proof}[Proof of (\ref{hb})] Due to the norm equivalence (\ref{ne}), it suffices to bound
$$
	\sup_{0\leq \tau \leq t} \big(\|u\|_{L^2}^2 +\|b\|_{L^2}^2) + \int_0^t \Big( \|\p_1u\|_{L^2}^2 + \|\na_hb\|_{L^2}^2 \Big)d\tau
$$
and
\begin{equation}\nonumber
\sup_{0 \leq \tau \leq t} \sum_{i = 1}^3 \big( \|\p_i^{3}\omega\|_{L^2}^2 + \|\p_i^{3}H\|_{L^2}^2 \big)+\sum_{i = 1}^3 \int_0^t \Big( \|\p_i^{3}\p_1\omega\|_{L^2}^2 + \|\p_i^3\na_h H\|_{L^2}^2 \Big)d\tau,
\end{equation}
where $\omega = \nabla\times u$ and $ H = \nabla\times b$ are the vorticity and the current
density, respectively. As aforementioned, $\|\om\|_{L^2}= \|\na u\|_{L^2}$ and $\|H\|_{L^2} = \|\na b\|_{L^2}$.

\vskip .1in
Taking the inner product  of $(u, b)$ with the first two equations of \eqref{mhd}, integrating by parts and using $\na\cdot u=\na\cdot b=0$, and then integrating in time,  we find
\begin{equation}\label{4eq1}
	\|u\|_{L^2}^2 +\|b\|_{L^2}^2+2 \int_0^t \Big( \|\p_1u\|_{L^2}^2 + \|\na_hb\|_{L^2}^2 \Big)d\tau {=} \|u_0\|_{L^2}^2 + \|b_0\|_{L^2}^2.
\end{equation}
Applying the operator $\nabla\times$ to \eqref{mhd}, we  obtain the system governing $(\omega,H)$,
\begin{equation}\label{5eq1}
\begin{cases}
\p_t \omega + u \cdot \nabla \omega - \omega\cdot\nabla u-     \partial_1^2 \omega = b \cdot \nabla H - H\cdot\nabla b + \partial_2 H,\\
\p_t H + \nabla\times(u\cdot\nabla b)-\Delta_h H = \nabla\times(b\cdot\nabla u)+\p_2\omega.
\end{cases}
\end{equation}
Applying $\p_i^{3}$ with $i=1,2,3$ to \eqref{5eq1} and taking the inner product of $(\p_i^{3}\omega,\p_i^{3}H)$ with the resulting equations, we have, after integration by parts,
\begin{equation}\label{5eq2}
\begin{split}
  &\frac{1}{2}\frac{d}{dt} \sum_{i=1}^3 \Big[ \| \p_i^{3}\omega\|_{L^2}^2 + \| \p_i^{3}H\|_{L^2}^2\Big] + \sum_{i=1}^3 \Big[ \| \p_i^{3}\p_1 \omega\|_{L^2}^2 + \| \p_i^{3}\na_h H\|_{L^2}^2  \Big]\\
  = & \sum_{i=1}^3 \int_{\mathbb{R}^3} \p_i^{3} \big[ - u \cdot \nabla \omega + \omega\cdot\nabla u + b \cdot \nabla H - H\cdot\nabla b + \partial_2 H \big] \cdot \p_i^{3} \omega dx \\
  & + \sum_{i=1}^3  \int_{\mathbb{R}^3} \p_i^{3}\big[ -\nabla\times(u\cdot\nabla b)+ \nabla\times(b\cdot\nabla u)+\p_2\omega \big]\cdot \p_i^{3} H dx \\
  =& \;  I_1+I_2+I_3+I_4+I_5+I_6+I_7+I_8.
\end{split}
\end{equation}
For the first term $I_1$, it can be written into three parts,
$$\int_{\mathbb{R}^3} \p_1^{3} ( - u \cdot \nabla \omega)\cdot \p_1^{3} \omega dx, \; \int_{\mathbb{R}^3} \p_2^{3} ( - u \cdot \nabla \omega)\cdot \p_2^{3} \omega dx,\;
\int_{\mathbb{R}^3} \p_3^{3} ( - u \cdot \nabla \omega)\cdot \p_3^{3} \omega dx.$$
The first and second parts above behave good and can be bounded easily,
\begin{equation}\nonumber
\begin{split}
& \int_0^t \int_{\mathbb{R}^3} \p_1^{3} ( - u \cdot \nabla \omega)\cdot \p_1^{3} \omega dx \;d\tau + \int_0^t \int_{\mathbb{R}^3} \p_2^{3} ( - u \cdot \nabla \omega)\cdot \p_2^{3} \omega dx \;d\tau \\
& \lesssim \sup_{0 \leq \tau \leq t} \|u(\tau)\|_{H^4} \int_0^t \|\na_h \omega\|_{\dot H^{2}}^2 \; d\tau \lesssim \mathscr{E}^{\frac{3}{2}} (t).
\end{split}
\end{equation}
We then turn to the hard term $\int_{\mathbb{R}^3} \p_3^{3} ( - u \cdot \nabla \omega)\cdot \p_3^{3} \omega dx $ in $I_1$. It's difficult to control since there is no dissipation in the $x_3$ direction. We further decompose it into three parts,
\begin{equation*}\label{5eq3}
\begin{split}
&\int_{\mathbb{R}^3} \p_3^{3} ( - u \cdot \nabla \omega)\cdot \p_3^{3} \omega dx \\
= & \sum_{k = 1}^{3}\int_{\mathbb{R}^3}  - \mathcal{C}_{3}^k \p_3^k u \cdot \nabla \p_3^{3-k} \omega\cdot \p_3^{3} \omega dx   \\
=  &  -\Big\{ \sum_{k=2}^{3} \mathcal{C}_{3}^k  \int_{\mathbb{R}^3}\p_3^k u_h \cdot \na_h\p_3^{3-k}\omega \cdot \p_3^{3}\omega dx + 3\int_{\mathbb{R}^3}\p_3 u_h \cdot \na_h
\p_3^{2}\omega \cdot \p_3^{3}\omega dx \Big\} \\
& -   \sum_{k=2}^{3} \mathcal{C}_{3}^k \int_{\mathbb{R}^3}\p_3^k u_3\p_3^{4-k}\omega \cdot \p_3^{3}\omega dx - 3\int_{\mathbb{R}^3}\p_3u_3\p_3^{3}\omega \cdot \p_3^{3}\omega dx \\
=& \; I_{11}+ I_{12}+ I_{13}.
\end{split}
\end{equation*}
By Lemma \ref{lem-anisotropic-est},
\begin{equation}\nonumber
\begin{split}
  {|I_{11}| } \lesssim & \sum_{k=2}^{3} \|\p_3^k u_h\|_{L^2}^\frac{1}{2} \|{\bf \p_2} \p_3^k u_h\|_{L^2}^\frac{1}{2} \|\na_h \p_3^{3-k} \omega\|_{L^2}^\frac{1}{2} \|{\bf \p_3} \na_h \p_3^{3-k} \omega\|_{L^2}^\frac{1}{2} \|\p_3^{3}\omega\|_{L^2}^\frac{1}{2}\|{\bf \p_1}\p_3^{3}\omega\|_{L^2}^\frac{1}{2} \\
  & + \|\p_3 u_h\|_{L^2}^\frac{1}{4} \|{\bf \p_2} \p_3 u_h\|_{L^2}^\frac{1}{4} \|{\bf \p_3} \p_3u_h\|_{L^2}^\frac{1}{4}\|{\bf \p_2\p_3} \p_3u_h\|_{L^2}^\frac{1}{4}\|\na_h \p_3^{2} \omega\|_{L^2}\|\p_3^{3}\omega\|_{L^2}^\frac{1}{2}\|{\bf \p_1}\p_3^{3}\omega\|_{L^2}^\frac{1}{2}.
\end{split}
\end{equation}
Integrating in time and applying H\"{o}lder's inequality yields
\begin{equation}\nonumber
 \int_0^t       {|I_{11}| } d\tau \lesssim \mathscr{E}^{\frac{3}{2}} (t).
\end{equation}
Similarly,
\begin{equation}\nonumber
\begin{split}
 {I_{12} }= & \sum_{k=2}^{3} \mathcal{C}_{3}^k \int_{\mathbb{R}^3}\p_3^{k-1} \nabla_h \cdot  u_h \p_3^{4-k}\omega \cdot \p_3^{3}\omega dx \\
  \lesssim & \sum_{k=2}^{3} \|\p_3^{k-1} \nabla_h \cdot u_h\|_{L^2}^\frac{1}{2} \|{\bf \p_3} \p_3^{k-1} \nabla_h \cdot u_h\|_{L^2}^\frac{1}{2}\\
    &\cdot \|\p_3^{4-k} \omega\|_{L^2}^\frac{1}{2}\|{\bf \p_2}\p_3^{4-k} \omega\|_{L^2}^\frac{1}{2}\|\p_3^{3}\omega\|_{L^2}^\frac{1}{2}\|{\bf \p_1}\p_3^{3}\omega\|_{L^2}^\frac{1}{2}.
\end{split}
\end{equation}
Therefore,
\begin{equation}\nonumber
 \int_0^t       | { I_{12} }  | d\tau \lesssim \mathscr{E}^{\frac{3}{2}} (t).
\end{equation}
By the divergence free condition $\nabla\cdot u=0$,
we can further split $ I_{13} $ into two parts,
\begin{equation}\nonumber
\begin{split}
  {I_{13} }= -\int_{\mathbb{R}^3}\p_1u_1\p_3^{3}\omega\cdot  \p_3^{3}\omega dx- \int_{\mathbb{R}^3}\p_2u_2\p_3^{3}\omega\cdot  \p_3^{3}\omega dx
  =  \; {I_{131} + I_{132}}.
\end{split}
\end{equation}
Using integration by parts and Lemma \ref{lem-anisotropic-est}, we have
\begin{equation}\nonumber
\begin{split}
\int_0^t { I_{131}}\; d\tau = & \; 2 \int_0^t  \int_{\mathbb{R}^3} u_1\p_3^{3}\omega \cdot \p_1\p_3^{3}\omega dx \\
\lesssim & \int_0^t  \|u_1\|_{L^2}^\frac{1}{4} \|{\bf \p_2} u_1\|_{L^2}^\frac{1}{4} \|{\bf \p_3} u_1\|_{L^2}^\frac{1}{4} \|{\bf \p_2 \p_3} u_1\|_{L^2}^\frac{1}{4} \\
&\qquad \cdot \|\p_3^{3}\omega\|_{L^2}^\frac{1}{2} \|{\bf \p_1} \p_3^{3}\omega\|_{L^2}^\frac{1}{2} \|\p_1\p_3^{3}\omega\|_{L^2} \;d\tau \\
\lesssim & \; \mathscr{E}^\frac{3}{2}(t).
\end{split}
\end{equation}
 However $I_{132}$ can not be similarly treated as $ I_{131}$. Integration by parts would generate
 $\p_2 \p_3^3 \om$, which can not be bounded by either $\mathscr{E}_1$ or $\mathscr{E}_2$.  We call this trouble the derivative loss problem. How to overcome the derivative loss problem is the main challenge of our proof.
 By creating several new techniques, we are able to deal with this type of terms. This is done in Proposition \ref{lem-important}. Therefore, by Proposition \ref{lem-important} we have
 $$
 \int_0^t  I_{132}\; d\tau \lesssim \mathscr{E}^{\frac{3}{2}} (0) + \mathscr{E}^{\frac{3}{2}} (t) + \mathscr{E}^2 (t).
 $$
Collecting all the upper bounds for various parts of $I_1$, we obtain
\begin{equation}\label{I1}
\int_0^t  |I_1(\tau) | d\tau \lesssim \mathscr{E}^{\frac{3}{2}} (0) + \mathscr{E}^{\frac{3}{2}} (t) + \mathscr{E}^2 (t).
\end{equation}

\vskip .1in
We turn to the second term $I_2$.
Naturally we divide it into the following three parts,
$$ \int_{\mathbb{R}^3} \p_1^{3}(\omega \cdot\nabla u)\cdot \p_1^{3}\omega dx, \quad  \int_{\mathbb{R}^3} \p_2^{3}(\omega \cdot\nabla u)\cdot \p_2^{3}\omega dx \quad \text{and} \quad \int_{\mathbb{R}^3} \p_3^{3}(\omega \cdot\nabla u)\cdot \p_3^{3}\omega dx. $$
The first two parts above can be controlled easily like before,
\begin{equation}\nonumber
  \begin{split}
    \int_0^t \int_{\mathbb{R}^3} \na_h^{3}(\omega \cdot\nabla u)\cdot \na_h^{3}\omega \; dx dt \lesssim \int_0^t \|u\|_{H^3} \|\na_h u\|_{H^{3}}\|\na_h^{3} \omega\|_{L^2} \; d\tau \lesssim \mathscr{E}^\frac{3}{2}(t).
  \end{split}
\end{equation}
The last part is further split into three terms as follows,
\begin{equation}\label{5eq4}
\begin{split}
\int_{\mathbb{R}^3} \p_3^{3}(\omega \cdot\nabla u)\cdot \p_3^{3}\omega dx = &\int_{\mathbb{R}^3} \p_3^{3}(\omega_1 \p_1 u)\cdot \p_3^{3}\omega dx + \int_{\mathbb{R}^3} \p_3^{3}(\omega_2 \p_2 u)\cdot \p_3^{3}\omega dx \\
& + \int_{\mathbb{R}^3} \p_3^{3}(\omega_3 \p_3 u)\cdot \p_3^{3}\omega dx \\
= & \; I_{21} + I_{22} + I_{23}.
\end{split}
\end{equation}
The estimate for $I_{21}$ is not difficult. By integration by parts,
\begin{equation}\nonumber
  \begin{split}
    & \int_0^t I_{21}(\tau) \; d\tau =  \sum_{k = 0}^{2} \mathcal{C}_{3}^k \int_0^t \int_{\mathbb{R}^3} \p_3^k \omega_1 \p_1 \p_3^{3-k} u \cdot \p_3^{3}\omega \; dx d\tau\\
    &\qquad\qquad\qquad + \int_0^t \int_{\mathbb{R}^3} \p_3^{3} \omega_1 \p_1 u\cdot  \p_3^{3}\omega \; dx d\tau \\
    \lesssim & \sum_{k = 0}^{ 2}\int_0^t \|\p_3^k \omega_1\|_{L^2}^\frac{1}{2} \|{\bf \p_2} \p_3^k \omega_1\|_{L^2}^\frac{1}{2}\|\p_1 \p_3^{3-k} u\|_{L^2}^\frac{1}{2} \|{\bf \p_3}\p_1 \p_3^{3-k} u\|_{L^2}^\frac{1}{2}\\
    &\qquad \quad \cdot \|\p_3^{3}\omega\|_{L^2}^\frac{1}{2}\|{\bf \p_1}\p_3^{3}\omega\|_{L^2}^\frac{1}{2} \; d\tau \\
    & + \int_0^t \| u \|_{L^2}^{\frac{1}{4}} \| {\bf \p_2} u \|_{L^2}^{\frac{1}{4}} \| {\bf \p_3} u \|_{L^2}^{\frac{1}{4}} \| {\bf \p_2 \p_3} u \|_{L^2}^{\frac{1}{4}}  \|\p_1 \p_3^{3}\omega \|_{L^2}\\
    &\qquad \quad \cdot \|\p_3^{3}\omega \|_{L^2}^{\frac{1}{2}} \|{\bf \p_1}\p_3^{3}\omega \|_{L^2}^{\frac{1}{2}} \; d\tau\\
    \lesssim & \; \mathscr{E}^\frac{3}{2}(t).
  \end{split}
\end{equation}
$I_{22}$ contains a difficult term that has to be dealt with by Proposition \ref{lem-important}.
By Lemma \ref{lem-anisotropic-est} and Proposition \ref{lem-important},
\begin{equation}\nonumber
  \begin{split}
    &\int_0^t I_{22}(\tau) \; d\tau \\
    = & {\sum_{k = 0}^{2}} \mathcal{C}_{3}^k \int_0^t \int_{\mathbb{R}^3} \p_3^k \omega_2 \p_2 \p_3^{3-k} u \cdot \p_3^{3}\omega \; dx d\tau + \int_0^t \int_{\mathbb{R}^3} \p_3^{3} \omega_2 \p_2 u \cdot \p_3^{3}\omega \; dx d\tau\\
    \lesssim & \int_0^t \|\omega_2\|_{H^2}^\frac{1}{2} \|{\bf \p_2} \omega\|_{H^2}^\frac{1}{2} \|\p_2 u\|_{H^3} \|\p_3^{3}\omega\|_{L^2}^\frac{1}{2} \|{\bf \p_1}\p_3^{3}\omega\|_{L^2}^\frac{1}{2}\;d\tau  +{\mathscr{E}^{\frac{3}{2}} (0) } + \mathscr{E}^\frac{3}{2}(t) + \mathscr{E}^2(t) \\
    \lesssim & \; \mathscr{E}^{\frac{3}{2}} (0) +\mathscr{E}^\frac{3}{2}(t) + \mathscr{E}^2(t).
  \end{split}
\end{equation}
Now we come to deal with the last part in \eqref{5eq4}, i.e.,
\begin{equation}\nonumber
I_{23} = \sum_{k=1}^{3}\mathcal{C}_{3}^k \int_{\mathbb{R}^3} \p_3^k \omega_3 \p_3^{4-k} u \cdot \p_3^{3}\omega \; dx + \int_{\mathbb{R}^3} \omega_3 \p_3^4 u \cdot \p_3^{3}\omega \; dx = I_{231} + I_{232}.
\end{equation}
Due to the divergence free condition of $\omega$, $I_{231}$ is easy to control. By $\nabla \cdot \omega = 0$,
\begin{equation}\nonumber
  \begin{split}
    I_{231} = & - \sum_{k=1}^{2}\mathcal{C}_{3}^k \int_{\mathbb{R}^3} \p_3^{k-1}(\p_1 \omega_1 + \p_2 \omega_2) \p_3^{4-k} u \cdot \p_3^{3}\omega \; dx\\
    & - \int_{\mathbb{R}^3} \p_3^{2}(\p_1 \omega_1 + \p_2 \omega_2) \p_3 u \cdot \p_3^{3}\omega \; dx \\
    \lesssim & \; \|\na_h \omega\|_{H^{1}}^\frac{1}{2} \|{\bf \p_3} \na_h \omega\|_{H^{1}}^\frac{1}{2} \|u\|_{H^{3}}^\frac{1}{2} \|{\bf \p_2} u\|_{H^{3}}^\frac{1}{2}
    \|\p_3^{3}\omega\|_{L^2}^\frac{1}{2}\|{\bf \p_1}\p_3^{3}\omega\|_{L^2}^\frac{1}{2} \\
    & + \|\p_3^{2} \na_h \omega\|_{L^2} \|\p_3 u\|_{L^2}^\frac{1}{4}\|{\bf \p_2}\p_3 u\|_{L^2}^\frac{1}{4}\|{\bf \p_3}\p_3 u\|_{L^2}^\frac{1}{4}\|{\bf \p_2 \p_3}\p_3 u\|_{L^2}^\frac{1}{4} \|\p_3^{3}\omega\|_{L^2}^\frac{1}{2}\|{\bf \p_1}\p_3^{3}\omega\|_{L^2}^\frac{1}{2}.
  \end{split}
\end{equation}
 $I_{232}$ can not be bounded in the similar way. We write $\omega_3 = \p_1 u_2 -\p_2 u_1$ in $I_{232}$,
\begin{equation}\nonumber
\int_0^t I_{232}(\tau) \; d\tau = \int_0^t \int_{\mathbb{R}^3} (\p_1u_2-\p_2u_1)\p_3^4 u\cdot  \p_3^{3}\omega dx d\tau.
\end{equation}
The estimate for $ \int_0^t \int_{\mathbb{R}^3} \p_1u_2\p_3^4 u \cdot \p_3^{3}\omega dx d\tau$ is just similar to $\int_0^t \int_{\mathbb{R}^3}\p_3^{3}\omega_1 \p_1 u \cdot \p_3^{3}\omega dx d\tau$ in $I_{21}$.
We can apply integration by parts and Lemma \ref{lem-anisotropic-est} to control it by $\mathscr{E}^\frac{3}{2}(t)$. It remains difficult to deal with $ \int_0^t \int_{\mathbb{R}^3} \p_2u_1\p_3^4 u\cdot  \p_3^{3}\omega dx d\tau$ due to the derivative loss problem. We split this term into three parts again.
\begin{equation}\label{uu}
\begin{split}
\int_0^t \int_{\mathbb{R}^3} \p_2u_1 \p_3^4 u \p_3^{3}\omega dx d\tau=& \int_0^t \int_{\mathbb{R}^3}\p_2u_1 \p_3^4 u_1 \p_3^{3}\omega_1 dxd\tau\\
& +\int_0^t \int_{\mathbb{R}^3}\p_2u_1 \p_3^4 u_2\p_3^{3}\omega_2 dx d\tau \\
& +\int_0^t \int_{\mathbb{R}^3}\p_2u_1 \p_3^4 u_3 \p_3^{3}\omega_3 dxd\tau.
\end{split}
\end{equation}
The last part $\int_0^t \int_{\mathbb{R}^3}\p_2u_1 \p_3^4 u_3 \p_3^{3}\omega_3 dxd\tau$ can be controlled by $\mathscr{E}^\frac{3}{2}(t)$ easily by making use of the divergence free property of $u$.
{ The process is simpler than that for $I_{231}$}.
  Now, we focus on the first two parts in \eqref{uu}. We  write them as
\begin{equation}\label{UU}
\begin{split}
& \int_0^t \int_{\mathbb{R}^3}\p_2u_1 \p_3^4 u_1 \p_3^{3}\omega_1 dx d\tau +\int_0^t \int_{\mathbb{R}^3}\p_2u_1 \p_3^4 u_2\p_3^{3}\omega_2 dx d\tau\\
= &  \int_0^t \int_{\mathbb{R}^3}\p_2u_1 \p_3^{3}\big(\p_3u_1-\p_1u_3 +\p_1u_3\big) \p_3^{3}\omega_1 dx d\tau \\
& +\int_0^t \int_{\mathbb{R}^3}\p_2u_1 \p_3^{3} \big(\p_3u_2-\p_2u_3+\p_2u_3 \big)\p_3^{3}\omega_2 dx d\tau\\
= & \int_0^t \int_{\mathbb{R}^3}\p_2u_1 \p_3^{3}\omega_2 \p_3^{3}\omega_1 dx  -\int_0^t  \int_{\mathbb{R}^3}\p_2u_1 \p_3^{3}\omega_1 \p_3^{3}\omega_1 dx d\tau\\
& + \int_0^t \int_{\mathbb{R}^3}\p_2u_1 \p_3^{3} \p_1u_3  \p_3^{3}\omega_1 dx+ \int_0^t \int_{\mathbb{R}^3}\p_2u_1 \p_3^{3} \p_2u_3  \p_3^{3}\omega_1 dx d\tau.
\end{split}
\end{equation}
The first two terms  on the right hand side of \eqref{UU} can be handled by Proposition \ref{lem-important} while  the rest two terms by Sobolev's inequality,
\begin{align*}
&\int_0^t \int_{\mathbb{R}^3}\p_2u_1 \p_3^{3}\omega_2 \p_3^{3}\omega_1 dx  -\int_0^t  \int_{\mathbb{R}^3}\p_2u_1 \p_3^{3}\omega_1 \p_3^{3}\omega_1 dx d\tau \lesssim \mathscr{E}^{\frac{3}{2}} (0) +\mathscr{E}^{\frac{3}{2}} (t) + \mathscr{E}^2 (t), \\
&\int_0^t \int_{\mathbb{R}^3}\p_2u_1 \p_3^{3} \p_1u_3  \p_3^{3}\omega_1 dx+ \int_0^t \int_{\mathbb{R}^3}\p_2u_1 \p_3^{3} \p_2u_3  \p_3^{3}\omega_1 dx d\tau \lesssim \mathscr{E}^{\frac{3}{2}} (t).
\end{align*}
 Taking all the inequalities above into consideration we then obtain
\begin{equation}\label{I2}
\int_0^t  |I_2(\tau) | d\tau \lesssim \mathscr{E}^{\frac{3}{2}} (0) +\mathscr{E}^{\frac{3}{2}} (t) + \mathscr{E}^2 (t).
\end{equation}
We shall combine the estimates of $I_3$ and $I_7$ to take advantage of cancellations.
Naturally $I_3 = \sum_{i=1}^3 \int_{\mathbb{R}^3} \p_i^{3} (b\cdot\nabla H)\cdot \p_i^{3}\omega dx $ is divided into the following three terms,
$$ \int_{\mathbb{R}^3} \p_1^{3} (b\cdot\nabla H)\cdot \p_1^{3}\omega dx, \quad \int_{\mathbb{R}^3} \p_2^{3} (b\cdot\nabla H)\cdot \p_2^{3}\omega dx \quad \text{and}\,\, \int_{\mathbb{R}^3} \p_3^{3} (b\cdot\nabla H)\cdot \p_3^{3}\omega dx .$$
Each term above can be split into two parts as follows.
$$ I_{31} + \tilde{I}_{31} = \sum_{k = 1}^{3}\mathcal{C}_{3}^k \int_{\mathbb{R}^3}\p_1^k b\cdot\nabla \p_1^{3-k}H \cdot \p_1^{3}\omega dx + \int_{\mathbb{R}^3} b\cdot\nabla \p_1^{3} H\cdot  \p_1^{3}\omega dx, $$
$$ I_{32} + \tilde{I}_{32} =\sum_{k = 1}^{3}\mathcal{C}_{3}^k \int_{\mathbb{R}^3}\p_2^k b\cdot\nabla \p_2^{3-k}H \cdot \p_2^{3}\omega dx + \int_{\mathbb{R}^3} b\cdot\nabla \p_2^{3} H \cdot \p_2^{3}\omega dx, $$
$$ I_{33} + \tilde{I}_{33} =\sum_{k = 1}^{3}\mathcal{C}_{3}^k \int_{\mathbb{R}^3}\p_3^k b\cdot\nabla \p_3^{3-k}H \cdot \p_3^{3}\omega dx + \int_{\mathbb{R}^3} b\cdot\nabla \p_3^{3} H \cdot \p_3^{3}\omega dx. $$
$I_{31}$ and $I_{32}$ have the good derivatives $\na_h$ and can be bounded directly,
\begin{equation}\nonumber
I_{31} + I_{32} \lesssim  \|\na_h b\|_{H^2} \|b\|_{H^4} \|\na_h \omega\|_{H^{2}}.
\end{equation}
To deal with $I_{33}$,  we also use the divergence free property $\nabla \cdot b = 0$ to write
\begin{equation}\nonumber
\begin{split}
I_{33} = & \sum_{k = 1}^{3}\mathcal{C}_{3}^k \int_{\mathbb{R}^3}\p_3^k b_h \cdot\nabla_h \p_3^{3-k}H \cdot \p_3^{3}\omega dx - \sum_{k = 1}^{3}\mathcal{C}_{3}^k \int_{\mathbb{R}^3}\p_3^{k-1}\nabla _h \cdot b_h\p_3^{4-k}H \cdot \p_3^{3}\omega dx.
\end{split}
\end{equation}
We have converted some of $\p_3$ into the good derivatives $\na_h$. By Lemma \ref{lem-anisotropic-est},
\begin{equation}\nonumber
  \begin{split}
  	 I_{33} \lesssim & (\|b_h\|_{H^{3}}^\frac{1}{2}\|{\bf \p_2} b_h\|_{H^{3}}^\frac{1}{2}\|\na_h H\|_{H^{2}}^\frac{1}{2}\|{\bf \p_3}\na_h H\|_{H^{2}}^\frac{1}{2}
  	{+\|\na_h b\|_{H^{2}}^\frac{1}{2}\|{\bf \p_3} \na_h b\|_{H^{2}}^\frac{1}{2}\| H\|_{H^{3}}^\frac{1}{2}
  		\|{\bf \p_2} H\|_{H^{3}}^\frac{1}{2}})\\
  	&\cdot\|\p_3^{3}\omega\|_{L^2}^\frac{1}{2}\|{\bf \p_1}\p_3^{3}\omega\|_{L^2}^\frac{1}{2}.
  \end{split}
\end{equation}
The remaining terms $\tilde{I}_{31}$, $\tilde{I}_{32}$ and $\tilde{I}_{33}$ can not be controlled directly,  but will be canceled
by the corresponding terms in  $I_7$. Now let's focus on $I_7= \sum_{i=1}^3 \int_{\mathbb{R}^3} \p_i^{3} [ \nabla \times(b\cdot\nabla u)]\cdot\p_i^{3}H dx $, we notice that
$$
\nabla \times(b\cdot\nabla u)=b\cdot\nabla \omega + \mathcal{R}
$$
where $\mathcal{R}$ stands for the vector with its $i$-th component given by
$$
\mathcal{R}_i = \sigma_{ijk} \p_j b\cdot \nabla u_k.
$$
Here $\sigma_{ijk}$ is the Levi-Cevita symbol,
\begin{equation}\label{sigma}
  \sigma_{ijk} = \left\{
                   \begin{array}{ll}
                     1, & \hbox{$ijk=123,231,312;$} \\
                     -1, & \hbox{$ijk=321,213,132;$} \\
                     0, & \hbox{otherwise.}
                   \end{array}
                 \right.
\end{equation}
Following  the process for $I_3$, we can split $I_7$ into the following nine parts,
\begin{equation}\nonumber
\begin{split}
   & I_{71} + \tilde{I}_{71} + \tilde{\tilde{I}}_{71} \\
  = &  \sum_{k = 1}^{3}\mathcal{C}_{3}^k\int_{\mathbb{R}^3} \p_1^k b\cdot\nabla \p_1^{3-k} \omega\cdot  \p_1^{3}H dx + \int_{\mathbb{R}^3} b\cdot\nabla \p_1^{3} \omega \cdot \p_1^{3}H dx + \int_{\mathbb{R}^3} \p_1^{3} ( \mathcal{R} )\cdot \p_1^{3}H dx,
\end{split}
\end{equation}
\begin{equation}\nonumber
\begin{split}
   & I_{72} + \tilde{I}_{72} + \tilde{\tilde{I}}_{72} \\
  = &  \sum_{k = 1}^{3}\mathcal{C}_{3}^k\int_{\mathbb{R}^3} \p_2^k b\cdot\nabla \p_2^{3-k} \omega \cdot \p_2^{3}H dx + \int_{\mathbb{R}^3} b\cdot\nabla \p_2^{3} \omega\cdot  \p_2^{3}H dx + \int_{\mathbb{R}^3} \p_2^{3} ( \mathcal{R} )\cdot \p_2^{3}H dx,
\end{split}
\end{equation}
\begin{equation}\nonumber
\begin{split}
   & I_{73} + \tilde{I}_{73} + \tilde{\tilde{I}}_{73} \\
  = &  \sum_{k = 1}^{3}\mathcal{C}_{3}^k\int_{\mathbb{R}^3} \p_3^k b\cdot\nabla \p_3^{3-k} \omega\cdot \p_3^{3}H dx + \int_{\mathbb{R}^3} b\cdot\nabla \p_3^{3} \omega \cdot \p_3^{3}H dx + \int_{\mathbb{R}^3} \p_3^{3} ( \mathcal{R} )\cdot \p_3^{3}H dx .
\end{split}
\end{equation}
By integration by parts and  $\na\cdot b=0$,
$$ \tilde{I}_{31} + \tilde{I}_{71}=0, \quad \tilde{I}_{32} + \tilde{I}_{72}=0 \quad \text{and} \quad  \tilde{I}_{33} + \tilde{I}_{73}=0.$$
$I_{71}$ and $I_{72}$ are easy to control while
$I_{73}$ can be written as
\begin{equation}\nonumber
\begin{split}
I_{73} = \sum_{k = 1}^{3}\mathcal{C}_{3}^k\int_{\mathbb{R}^3} \p_3^k b_h \cdot \nabla_h \p_3^{3-k} \omega\cdot \p_3^{3}H dx - \sum_{k = 1}^{3}\mathcal{C}_{3}^k\int_{\mathbb{R}^3} \p_3^{k-1}\nabla_h \cdot b_h \p_3^{4-k} \omega\cdot \p_3^{3}H dx.
\end{split}
\end{equation}
Now $I_{71}, I_{72}$ and $I_{73}$ all contain good derivatives $\na_h$ and can be handled exactly
like $I_{31}, I_{32}$ and $I_{33}$.
We omit the repeated details.
 Let's focus on the remaining terms containing $\mathcal{R}$, namely $\tilde{\tilde{I}}_{71}$, $\tilde{\tilde{I}}_{72}$ and $\tilde{\tilde{I}}_{73}$.
By H\"{o}lder's inequality and Sobolev embedding theorem,
\begin{equation}\nonumber
  \begin{split}
    \tilde{\tilde{I}}_{71} + \tilde{\tilde{I}}_{72} \lesssim {\big(} \|\na_h b\|_{H^{3}} \|u\|_{H^4} + \|\na_h u\|_{H^{3}} \|b\|_{H^4} {\big)}\|\na_h H\|_{H^{3}}.
  \end{split}
\end{equation}
$\tilde{\tilde{I}}_{73}$ is more complex and is further split  into two parts,
\begin{equation}\nonumber
  \begin{split}
    \tilde{\tilde{I}}_{73} = & \sum_{i=1}^3\int_{\mathbb{R}^3} \p_3^{3} (\sigma_{ijk} \p_j b_h \cdot \nabla_h u_k) \p_3^{3}H_i \;dx  + \int_{\mathbb{R}^3} \p_3^{3} (\sigma_{3jk} \p_j b_3 \p_3  u_k) \p_3^{3}H_3 \;dx  \\
    &+ \sum_{i \neq 3} \int_{\mathbb{R}^3} \p_3^{3} (\sigma_{ijk} \p_j b_3 \p_3 u_k) \p_3^{3}H_i \;dx = \tilde{\tilde{I}}_{731} + \tilde{\tilde{I}}_{732}.
  \end{split}
\end{equation}
Noticing $H_3 = \p_1 b_2 - \p_2 b_1$ and applying Lemma \ref{lem-anisotropic-est},  we can bound $\tilde{\tilde{I}}_{731}$ by
\begin{equation}\nonumber
  \begin{split}
    \tilde{\tilde{I}}_{731} \lesssim & \|\na_h u\|_{H^{3}} \|b_h\|_{H^{3}}^\frac{1}{4}\|{\bf \p_2}b_h\|_{H^{3}}^\frac{1}{4} \|{\bf \p_3}b_h\|_{H^{3}}^\frac{1}{4}
    \|{\bf \p_2\p_3}b_h\|_{H^{3}}^\frac{1}{4}   \|\p_3^{3}H\|_{L^2}^\frac{1}{2} \|{\bf \p_1}\p_3^{3} H\|_{L^2}^\frac{1}{2} \\
    & + {\|\na_h u\|_{L^{2}}^\frac{1}{2}\|{\bf \p_3} \na_h u\|_{L^{2}}^\frac{1}{2}} \|b_h\|_{H^4}^\frac{1}{2}\|{\bf \p_2}b_h\|_{H^4}^\frac{1}{2}\|\p_3^{3} H\|_{L^2}^\frac{1}{2}
    \|{\bf \p_1}\p_3^{3} H\|_{L^2}^\frac{1}{2}\\
    & + \|b_3 \|_{H^4}^\frac{1}{2} \|{\bf \p_2} b_3\|_{H^4}^\frac{1}{2} \|u\|_{H^4}^\frac{1}{2} \|{\bf \p_1} u\|_{H^4}^\frac{1}{2} \|\na_h b\|_{H^{3}}^\frac{1}{2} \|{\bf \p_3} \na_h b\|_{H^{3}}^\frac{1}{2}.
  \end{split}
\end{equation}
By the definition of $\sigma_{ijk}$, if $i\neq 3$, we will have $j=3$ or $k=3$. Indeed,
\begin{equation}\nonumber
  \begin{split}
    \tilde{\tilde{I}}_{732} \lesssim & \|\p_3 b_3\|_{H^{3}}^\frac{1}{2}\|{\bf \p_3}\p_3 b_3\|_{H^{3}}^\frac{1}{2}\|u\|_{H^4}^\frac{1}{2}\|{\bf \p_1}u\|_{H^4}^\frac{1}{2}
    \|\p_3^{3}H\|_{L^2}^\frac{1}{2}\|{\bf \p_2}\p_3^{3}H\|_{L^2}^\frac{1}{2} \\
    & +  \|\p_3 u_3\|_{H^{2}}^\frac{1}{2}\|{\bf \p_3}\p_3 u_3\|_{H^{2}}^\frac{1}{2}\|b_3\|_{H^4}^\frac{1}{2}\|{\bf \p_1}b_3\|_{H^4}^\frac{1}{2}
    \|\p_3^{3}H\|_{L^2}^\frac{1}{2}\|{\bf \p_2}\p_3^{3}H\|_{L^2}^\frac{1}{2} \\
    & + \|\p_3 u_3 \|_{H^{3}} \|b_3\|_{H^1}^\frac{1}{4}\|{\bf \p_1}b_3\|_{H^1}^\frac{1}{4} \|{\bf \p_3}b_3\|_{H^1}^\frac{1}{4}\|{\bf \p_1 \p_3}b_3\|_{H^1}^\frac{1}{4}\|\p_3^{3}H\|_{L^2}^\frac{1}{2}\|{\bf \p_2}\p_3^{3}H\|_{L^2}^\frac{1}{2}.
  \end{split}
\end{equation}
Collecting the upper bounds for $I_3$ and $I_7$ above, we find
\begin{equation}\label{I3I7}
\int_0^t  |I_3(\tau) + I_7(\tau) |    d\tau \lesssim \mathscr{E}^{\frac{3}{2}} (t).
\end{equation}
Next we deal with $I_4$, which is naturally divided into the following three terms
$$
- \int_{\mathbb{R}^3} \p_1^3(H\cdot\nabla b)\cdot \p_1^3\omega dx, \quad - \int_{\mathbb{R}^3} \p_2^3 (H\cdot\nabla b)\cdot \p_2^3\omega dx, \quad - \int_{\mathbb{R}^3} \p_3^3 (H\cdot\nabla b)\cdot \p_3^3\omega dx.
$$
The first two terms already contain the good derivatives $\na_h$ and can be bounded directly. The third term is further decomposed into three terms
\begin{equation}\nonumber
\begin{split}
 - \int_{\mathbb{R}^3} \p_3^{3} (H\cdot\nabla b)\cdot \p_3^{3}\omega dx =&  - \int_{\mathbb{R}^3} \p_3^{3} (H_1 \p_1 b)\cdot \p_3^{3}\omega dx  - \int_{\mathbb{R}^3} \p_3^{3} (H_2 \p_2 b)\cdot \p_3^{3}\omega dx \\
 &- \int_{\mathbb{R}^3} \p_3^{3} (H_3 \p_3 b)\cdot \p_3^{3}\omega dx.
\end{split}
\end{equation}
The first two terms above have either $\p_1 b$ or $\p_2b$ and can thus be bounded
by applying Lemma \ref{lem-anisotropic-est}.  The third term involves $H_3 = \p_1b_2-\p_2b_1$ and thus also
contains the good horizontal derivatives on $b$. Therefore, all of them admit suitable upper bounds and
\begin{equation}\label{I4}
\int_0^t  |I_4(\tau)  |    d\tau \lesssim \mathscr{E}^{\frac{3}{2}} (t).
\end{equation}
$I_5$ and $I_8$ cancel each other by integration by parts,
\begin{equation}\label{I5I8}
I_5 + I_8 = \sum_{i=1}^3  \Big[ \int_{\mathbb{R}^3} \p_i^{3}\p_2H \cdot \p_i^{3} \omega dx + \int_{\mathbb{R}^3} \p_i^{3}\p_2 \omega\cdot  \p_i^{3} H dx \Big] =0.
\end{equation}
It remains to deal with $I_6$. As in $I_7$, we can write $\nabla \times (u\cdot \nabla b) = u \cdot \nabla H +  \tilde {\mathcal{R}}$, where $\tilde {\mathcal{R}_i} = \sigma_{ijk}\p_j u \cdot \nabla b_k$ and $\sigma_{ijk}$ is defined in \eqref{sigma}. Thus,
\begin{equation}\nonumber
\begin{split}
 I_6 = & -\sum_{k = 1}^{3}\mathcal{C}_{3}^k\int_{\mathbb{R}^3} \p_1^k u \cdot \nabla \p_1^{3-k}H\cdot  \p_1^{3}H dx - \int_{\mathbb{R}^3} \p_1^{3}(\tilde {\mathcal{R}}) \cdot \p_1^{3} H dx\\
 & -\sum_{k = 1}^{3}\mathcal{C}_{3}^k\int_{\mathbb{R}^3} \p_2^k u \cdot \nabla \p_2^{3-k}H \cdot \p_2^{3}H dx - \int_{\mathbb{R}^3} \p_2^{3}(\tilde {\mathcal{R}})\cdot  \p_2^{3} H dx\\
 & -\sum_{k = 1}^{3}\mathcal{C}_{3}^k\int_{\mathbb{R}^3} \p_3^k u_h \cdot \nabla_h \p_3^{3-k}H \cdot \p_3^{3}H dx -\sum_{k = 1}^{3}\mathcal{C}_{3}^k\int_{\mathbb{R}^3} \p_3^k u_3  \p_3^{4-k}H \cdot \p_3^{3}H dx \\
 & - \int_{\mathbb{R}^3} \p_3^{3}(\tilde {\mathcal{R}}) \cdot \p_3^{3} H dx
\\
= & \; I_{61}+I_{62} + I_{63} + I_{64}.
\end{split}
\end{equation}
By H\"{o}lder's inequality and Sobolev embedding theorem,
\begin{equation}\nonumber
  \begin{split}
    I_{61}+I_{62} \lesssim& \|\na_h u\|_{H^{3}} \|H\|_{H^3} \|\na_h H\|_{H^{3}}\\
    & + \big(\|\na_h u\|_{H^3} \|b\|_{H^4} + {\|u\|_{H^2}} \|\na_h b\|_{H^3} \big) \|\na_h H\|_{H^{3}}.
  \end{split}
\end{equation}
The estimate for $I_{63}$ is similar to that for $I_{73}$,
\begin{equation}\nonumber
  \begin{split}
    I_{63} \lesssim & \|u_h\|_{H^{3}}^\frac{1}{2} \|{\bf \p_1} u_h\|_{H^{3}}^\frac{1}{2}\|\na_h H\|_{H^{2}}^\frac{1}{2} \|{\bf \p_3} \na_h H\|_{H^{2}}^\frac{1}{2} \|\p_3^{3} H\|_{L^2}^\frac{1}{2} \|{\bf \p_2} \p_3^{3} H\|_{L^2}^\frac{1}{2} \\
    & + \|\nabla_h \cdot u_h\|_{H^{2}}^\frac{1}{2} \|{\bf \p_3} \nabla_h \cdot u_h\|_{H^{2}}^\frac{1}{2} \|H\|_{H^{3}}  \|{\bf \na_h} H\|_{H^{3}}.
  \end{split}
\end{equation}
$I_{64}$ is bounded similarly as $\tilde{\tilde{I}}_{73}$
\begin{equation}\nonumber
\begin{split}
  I_{64}  \lesssim & \; \|u_h\|_{H^4}^\frac{1}{2} \|{\bf \p_1} u_h\|_{H^4}^\frac{1}{2} \|\na_h b\|_{H^3}^\frac{1}{2} \|{\bf \p_3} \na_h b\|_{H^3}^\frac{1}{2} \|\p_3^{3}H\|_{L^2}^\frac{1}{2}\|{\bf \p_2}\p_3^{3}H\|_{L^2}^\frac{1}{2}\\
   &+ \|u_3 \|_{H^4}^\frac{1}{2} \|{\bf \p_1} u_3\|_{H^4}^\frac{1}{2} \|b\|_{H^4}^\frac{1}{2} \|{\bf \p_2} b\|_{H^4}^\frac{1}{2} \|\na_h b\|_{H^{3}}^\frac{1}{2} \|{\bf \p_3} \na_h b\|_{H^{3}}^\frac{1}{2} \\
  & +  \|\p_3 b_3\|_{H^{3}}^\frac{1}{2}\|{\bf \p_3}\p_3 b_3\|_{H^{3}}^\frac{1}{2}\|u\|_{H^4}^\frac{1}{2}\|{\bf \p_1}u\|_{H^4}^\frac{1}{2}
    \|\p_3^{3}H\|_{L^2}^\frac{1}{2}\|{\bf \p_2}\p_3^{3}H\|_{L^2}^\frac{1}{2} \\
    & +  \|\p_3 u_3\|_{H^{2}}^\frac{1}{2}\|{\bf \p_3}\p_3 u_3\|_{H^{2}}^\frac{1}{2}\|b\|_{H^4}^\frac{1}{2}\|{\bf \p_1}b\|_{H^4}^\frac{1}{2}
    \|\p_3^{3}H\|_{L^2}^\frac{1}{2}\|{\bf \p_2}\p_3^{3}H\|_{L^2}^\frac{1}{2} \\
    & + \|\p_3 u_3 \|_{H^{3}} \|b\|_{H^1}^\frac{1}{4}\|{\bf \p_1}b\|_{H^1}^\frac{1}{4} \|{\bf \p_3}b\|_{H^1}^\frac{1}{4}\|{\bf \p_1 \p_3}b\|_{H^1}^\frac{1}{4}\|\p_3^{3}H\|_{L^2}^\frac{1}{2}\|{\bf \p_2}\p_3^{3}H\|_{L^2}^\frac{1}{2}.
\end{split}
\end{equation}
Therefore,
\begin{equation}\label{I6}
\int_0^t  |I_6(\tau)  |    d\tau \lesssim \mathscr{E}^{\frac{3}{2}} (t).
\end{equation}
Integrating \eqref{5eq2} in time on the interval $[0,t]$ and invoking
the upper bounds in \eqref{I1}, \eqref{I2}, \eqref{I3I7}, \eqref{I4}, \eqref{I5I8} and \eqref{I6}, we obtain
\begin{equation}\label{5eq5}
\sup_{{0\le\tau\le t}} \big( \|u\|_{\dot H^4}^2 + \|b\|_{\dot H^4}^2 \big)+\int_0^t \Big( \|\p_1u\|_{\dot H^4}^2 + \|\na_h b\|_{\dot H^4}^2 \Big) d\tau    \lesssim \mathscr{E}(0) + \mathscr{E}^{\frac{3}{2}} (0) +\mathscr{E} ^{\frac{3}{2}}(t)+ \mathscr{E}^2(t).
\end{equation}
Adding (\ref{4eq1}) and (\ref{5eq5}) yields the desired inequality in (\ref{hb}), namely
$$
		\mathscr{E}_1(t)\lesssim \mathscr{E}(0) + \mathscr{E}^{\frac{3}{2}} (0) +\mathscr{E} ^{\frac{3}{2}}(t)+ \mathscr{E}^2(t).
$$
This finishes the proof.
\end{proof}

\vskip .3in
\section{Proof of Proposition \ref{lem-important}}
\label{proof-lem}

This section is devoted to the proof of Proposition \ref{lem-important}, which provides suitable upper bounds
for the time integral of the interaction terms. We have used Proposition \ref{lem-important} extensively in
the crucial energy estimates in the previous sections.

\vskip .1in
The proof of Proposition \ref{lem-important} deals with some of the most difficult terms emanating from
the velocity nonlinearity. We take out two of the wildest terms
and deal with them in the following lemma. We will state and prove this lemma, and then prove
Proposition \ref{lem-important}.

\begin{lemma}\label{lem-assistant}
	Let $(u, b)\in H^4$ be a solution of (\ref{mhd}). Let $\omega = \nabla \times u$ and $H = \nabla \times b$
	be the corresponding vorticity and current density, respectively. Let $\mathscr{E}(t)$ be defined as in  \eqref{ed}. Then, for all $i,j,k \in \{1,2,3\}$,
	$$ \left| \int_0^t \int_{\mathbb{R}^3} b_j \p_t(\p_3^{3}\omega_i \p_3^{3}\omega_k) \; dxd\tau \right|,  \,
	\left| \int_0^t \int_{\mathbb{R}^3} \p_2 u_j \p_t(\p_3^{3}\omega_i\p_3^{3}\omega_k) \;dx d\tau \right|    \lesssim {\mathscr{E}^{\frac{3}{2}}(0)}+ \mathscr{E}^{\frac{3}{2}}(t) + \mathscr{E}^2 (t).$$
\end{lemma}

\begin{proof}[Proof of Lemma \ref{lem-assistant}]
	
	We will deal with the two terms above simultaneously. By the equation of $\omega$ in (\ref{5eq1}),
	\begin{equation}\label{6eq2}
		\begin{split}
			& - \p_t(\p_3^{3}\omega_i\p_3^{3}\omega_k)\\
			= \;& \p_3^{3}\omega_i \p_3^{3}\big( u\cdot\nabla\omega_k - \omega\cdot\nabla u_k -\p_1^2 \omega_k -b\cdot\nabla H_k +H \cdot\nabla b_k -\p_2 H_k \big) \\
			\;& + \p_3^{3}\omega_k \p_3^{3}\big( u\cdot\nabla\omega_i - \omega\cdot\nabla u_i -\p_1^2 \omega_i -b\cdot\nabla H_i +H \cdot\nabla b_i -\p_2 H_i \big).
		\end{split}
	\end{equation}
Multiplying (\ref{6eq2}) by $b_j$ or $\p_2 u_j$ and integrating in space, we can write
\begin{equation}\label{J12345}
- \int_{\mathbb{R}^3} \p_t(\p_3^{3}\omega_i\p_3^{3}\omega_k) (b_j|\p_2 u_j )dx
=J_1 + J_2 +J_3+ J_4+J_5,
\end{equation}
where the notation $(b_j|\p_2 u_j)$  stands for either $b_j$ or $ \p_2 u_j$ and we will use it throughout
the rest of the proof. The explicit expression for $J_1 \thicksim J_5$ is shown below. We first deal with $J_1$ given by
	\begin{equation}\nonumber
		J_1 = \int_{\mathbb{R}^3} \big[\p_3^{3}\omega_i \p_3^{3}( u\cdot\nabla\omega_k) + \p_3^{3}\omega_k \p_3^{3} ( u\cdot\nabla\omega_i) \big] (b_j|\p_2 u_j )dx.
	\end{equation}
The highest order norms of $\omega$ in $J_1$, labeled as $J_{11}$, can be dealt with using Lemma \ref{lem-anisotropic-est},
	\begin{equation}\nonumber
		\begin{split}
			J_{11} = &\int_{\mathbb{R}^3} \Big[ \p_3^{3}\omega_i u\cdot\nabla \p_3^{3}\omega_k + \p_3^{3}\omega_k u\cdot\nabla \p_3^{3}\omega_i \Big] (b_j|\p_2 u_j )dx \\
			=& \int_{\mathbb{R}^3} u\cdot\nabla \big( \p_3^{3}\omega_i\p_3^{3}\omega_k \big)(b_j|\p_2 u_j )dx \\
			=& - \int_{\mathbb{R}^3} \p_3^{3}\omega_i\p_3^{3}\omega_k u\cdot\nabla (b_j|\p_2 u_j) dx \\
			\lesssim & \; \| {\bf \p_1} \omega\|_{H^{3}} \|\omega\|_{H^{3}} \|{\bf \p_2} u\|_{H^1}^{\frac{1}{2}} \|u\|_{H^1}^{\frac{1}{2}} \|{\bf \p_2} (b|\p_2 u)\|_{H^2}^{\frac{1}{2}} \| (b|\p_2 u)\|_{H^2}^{\frac{1}{2}}.
		\end{split}
	\end{equation}
The remaining parts in $J_1$ can be written as
	\begin{equation}\nonumber
		\begin{split}
			J_{12} =& \sum_{l = 1}^{3} \mathcal{C}_{3}^l\int_{\mathbb{R}^3} \big[\p_3^{3}\omega_i \p_3^{l} u\cdot\nabla \p_3^{3-l} \omega_k + \p_3^{3}\omega_k \p_3^{l} u\cdot\nabla \p_3^{3-l} \omega_i \big](b_j|\p_2 u_j )dx \\
			= &\sum_{l = 1}^{2} \mathcal{C}_{3}^l\int_{\mathbb{R}^3} \big[\p_3^{3}\omega_i \p_3^{l} u\cdot\nabla \p_3^{3-l} \omega_k + \p_3^{3}\omega_k \p_3^{l} u\cdot\nabla \p_3^{3-l} \omega_i \big](b_j|\p_2 u_j )dx \\
			& + \int_{\mathbb{R}^3} \big[\p_3^{3}\omega_i \p_3^{3} u\cdot\nabla \omega_k + \p_3^{3}\omega_k \p_3^{3} u\cdot\nabla \omega_i \big](b_j|\p_2 u_j )dx,
		\end{split}
	\end{equation}
which is easily controlled by
	\begin{equation}\nonumber
		\begin{split}
			&\|\p_3^{3} \omega\|_{L^2}^\frac{1}{2}\|{\bf \p_1}\p_3^{3} \omega\|_{L^2}^\frac{1}{2}
			\|u\|_{H^4}^\frac{1}{2}\|{\bf \p_1} u\|_{H^4}^\frac{1}{2}
			\|u\|_{H^3}^\frac{1}{2}\|{\bf \p_2} u\|_{H^3}^\frac{1}{2} \| (b|\p_2 u)\|_{H^1}^{\frac{1}{2}} \| {\bf \p_2}(b|\p_2 u)\|_{H^1}^{\frac{1}{2}}.
		\end{split}
	\end{equation}
$J_2$ is defined as and bounded by
	\begin{equation}\nonumber
		\begin{split}
			J_2 &= -\int_{\mathbb{R}^3} \p_3^{3}\omega_i \p_3^{3}( \omega \cdot\nabla u_k)(b_j|\p_2 u_j )dx -\int_{\mathbb{R}^3}\p_3^{3}\omega_k \p_3^{3}( \omega \cdot\nabla u_i)(b_j|\p_2 u_j )dx\\
			& \lesssim\|\p_3^{3} \omega\|_{L^2}^\frac{1}{2}\|{\bf \p_1}\p_3^{3} \omega\|_{L^2}^\frac{1}{2} \|u\|_{H^4}^\frac{1}{2}\|{\bf \p_1}u\|_{H^4}^\frac{1}{2}
			\|u\|_{H^{2}}^\frac{1}{4}\|{\bf \p_2}u\|_{H^{2}}^\frac{1}{4}\|{\bf \p_3}u\|_{H^{2}}^\frac{1}{4}\|{\bf \p_2 \p_3}u\|_{H^{2}}^\frac{1}{4} \\
			&\quad \cdot \| (b|\p_2 u)\|_{L^2}^{\frac{1}{4}} \| {\bf \p_2}(b|\p_2 u)\|_{L^2}^{\frac{1}{4}}
			\| {\bf \p_3}(b|\p_2 u)\|_{L^2}^{\frac{1}{4}} \| {\bf \p_2 \p_3}(b|\p_2 u)\|_{L^2}^{\frac{1}{4}}.
		\end{split}
	\end{equation}
By integration by parts,
	\begin{equation}\nonumber
		\begin{split}
			J_3=&\int_{\mathbb{R}^3} \big[-\p_3^{3}\omega_i \p_3^{3} \p_1^2 \omega_k -\p_3^{3}\omega_k \p_3^{3} \p_1^2 \omega_i\big] (b_j|\p_2 u_j ) \;dx\\
			=& \;  2 \int_{\mathbb{R}^3} \p_1\p_3^{3}\omega_i \p_1 \p_3^{3}\omega_k (b_j|\p_2 u_j )dx \\
			&+ \int_{\mathbb{R}^3} \big[\p_3^{3}\omega_i \p_1 \p_3^{3}\omega_k + \p_3^{3}\omega_k \p_1 \p_3^{3}\omega_i\big]\p_1(b_j|\p_2 u_j )dx \\
			\lesssim & \; \| \p_1 \omega \|_{H^{3}}^2 \| (b|\p_2 u) \|_{H^2} + \|\omega\|_{H^{3}}\|\p_1\omega\|_{H^{3}} \|\p_1(b|\p_2 u)\|_{H^2}.
		\end{split}
	\end{equation}
$J_4$ is given by
$$ J_4 =\int_{\mathbb{R}^3} \big[ -\p_3^{3}\omega_i \p_3^{3}( b\cdot\nabla H_k) -\p_3^{3}\omega_k \p_3^{3} ( b\cdot\nabla H_i) \big] (b_j|\p_2 u_j) \; dx . $$
This is a difficult term. First we separate $J_4$ into two parts,
\begin{equation}\nonumber
	\begin{split}
J_4 = &\int_{\mathbb{R}^3}\big[ -\p_3^{3}\omega_i b\cdot\nabla \p_3^{3}H_k  - \p_3^{3}\omega_k b\cdot\nabla\p_3^{3} H_i\big] (b_j | \p_2 u_j)\; dx\\
& + \sum_{l = 1}^{3} \mathcal{C}_{3}^l\int_{\mathbb{R}^3} \big[-\p_3^{3}\omega_i \p_3^{l} b\cdot\nabla \p_3^{3-l} H_k - \p_3^{3}\omega_k \p_3^{l} b\cdot\nabla \p_3^{3-l} H_i \big](b_j|\p_2 u_j )dx\\
=& J_{41} + J_{42}.
	\end{split}
\end{equation}
$J_{42}$ can be bounded directly.  By Lemma 2.1,
\begin{equation}\label{j42b}
\begin{split}	
|J_{42}| {\lesssim} & \| \omega \|_{H^3}^{\frac{1}{2}}  \| {\bf \p_1} \omega \|_{H^3}^{\frac{1}{2}}\| H \|_{H^3}^{\frac{1}{2}}  \| {\bf \p_1} H \|_{H^3}^{\frac{1}{2}}\| b \|_{H^4}^{\frac{1}{2}}  \| {\bf \p_2} b \|_{H^4}^{\frac{1}{2}} \\
 &\cdot  \| (b| \p_2 u)\|_{H^1}^{\frac{1}{2}}  \| {\bf \p_2} (b| \p_2 u) \|_{H^1}^{\frac{1}{2}}.
	\end{split}
\end{equation}
The estimate for $J_{41}$ is at the core of this section. $J_1$, $J_2$ and $J_3$ are symmetric in the sense that, when we switch $i$ and $k$ in any one of these terms, they remain the same. As we have seen in the estimates above,
terms with symmetric structure are relatively easy to deal with.  However, $J_{41}$ is not symmetric
and we can no longer make easy
cancellations. To overcome this essential difficulty, we construct some artificial symmetry to take full advantage of the cancellations.
	\begin{equation}\label{6eq3}
		\begin{split}
			J_{41} = &\int_{\mathbb{R}^3}\big[ -\p_3^{3}\omega_i b\cdot\nabla \p_3^{3}H_k  - \p_3^{3}\omega_k b\cdot\nabla\p_3^{3} H_i\big] (b_j | \p_2 u_j)\; dx\\
			= & \int_{\mathbb{R}^3} \big[\p_3^{3} H_i b\cdot\nabla \p_3^{3}\omega_k + \p_3^{3} H_k b\cdot\nabla \p_3^{3}\omega_i \big] (b_j | \p_2 u_j) \; dx\\
			&-\int_{\mathbb{R}^3}\big[ b\cdot\nabla (\p_3^{3}\omega_i\p_3^{3} H_k) + b\cdot\nabla (\p_3^{3}\omega_k\p_3^{3} H_i)\big](b_j | \p_2 u_j) \; dx \\
			= & \; J_{411} + J_{412}.
		\end{split}
	\end{equation}
$J_{412}$  can be handled through integration by parts and Lemma \ref{lem-anisotropic-est},
	\begin{equation}\label{j412b}
		\begin{split}
			J_{412}= &  \int_{\mathbb{R}^3} \big[  \p_3^{3}\omega_i\p_3^{3} H_k +\p_3^{3}\omega_k\p_3^{3} H_i \big] b\cdot\nabla(b_j|\p_2 u_j )  dx  \\
			\lesssim & \; \|\p_3^{3} \omega\|_{L^2}^\frac{1}{2}\|{\bf \p_1} \p_3^{3} \omega\|_{L^2}^\frac{1}{2}\|\p_3^{3} H\|_{L^2}^\frac{1}{2}\|{\bf \p_1} \p_3^{3} H\|_{L^2}^\frac{1}{2}\|b\|_{H^1}^\frac{1}{2}\|{\bf \p_2}b\|_{H^1}^\frac{1}{2}\\
			&\quad \cdot \|(b|\p_2 u)\|_{H^2}^\frac{1}{2}\|{\bf \p_2}(b|\p_2 u)\|_{H^2}^\frac{1}{2}.
		\end{split}
	\end{equation}
$J_{411}$ is extremely difficult.  As our first step, we invoke the equation of $H$ in (\ref{5eq1})
	\begin{equation}\label{6eq4}
		\begin{split}
			& \p_t H_{k} + u\cdot\nabla H_k + \sum_{p=1}^3 \nabla u_p \times \p_p b_k -\Delta_h H_k \\
			& = b\cdot\nabla \omega_k
			+ \sum_{p=1}^3 \nabla b_p \times \p_p u_k +\p_2 \omega_k .
		\end{split}
	\end{equation}
where we have used the simple identities
	\begin{equation}\nonumber
		\begin{split}
			&\nabla \times (b\cdot\nabla u) = b\cdot\nabla \omega + \sum_{p=1}^3 \nabla b_p \times \p_p u, \\
			&\nabla \times (u\cdot\nabla b) = u \cdot\nabla H + \sum_{p=1}^3 \nabla u_p \times \p_p b.
		\end{split}
	\end{equation}
Applying $\p_3^{3}$ to \eqref{6eq4} then yields
	\begin{equation}\label{6eq5}
		\begin{split}
			& (\p_3^{3} H_k)_t + \p_3^{3} (u\cdot\nabla H_k) +\p_3^{3}(\sum_{p=1}^3 \nabla u_p \times \p_p b_k) -\p_3^{3}\Delta_h H_k\\
			=&  \; b\cdot\nabla \p_3^{3} \omega_k + \sum_{l = 1}^{3} \mathcal{C}_{3}^l\p_3^l b\cdot\nabla \p_3^{3-l}\omega_k + \p_3^{3}( \sum_{p=1}^3 \nabla b_p\times \p_p u_k ) + \p_3^{3} \p_2 \omega_k.
		\end{split}
	\end{equation}
Similarly,
	\begin{equation}\label{6eq6}
		\begin{split}
			& (\p_3^{3} H_i)_t + \p_3^{3} (u\cdot\nabla H_i) +\p_3^{3}(\sum_{p=1}^3 \nabla u_p \times \p_p b_i) -\p_3^{3}\Delta_h H_i\\
			=&  \; b\cdot\nabla \p_3^{3} \omega_i + \sum_{l = 1}^{3} \mathcal{C}_{3}^l\p_3^l b\cdot\nabla \p_3^{3-l}\omega_i + \p_3^{3}( \sum_{p=1}^3 \nabla b_p\times \p_p u_i ) + \p_3^{3} \p_2 \omega_i .
		\end{split}
	\end{equation}
	Multiplying \eqref{6eq5} by $\p_3^{3} H_i$ and \eqref{6eq6} by $\p_3^{3} H_k$ , and summing them up, there holds
	\begin{equation}\label{6eq7}
		\begin{split}
			& \p_3^{3} H_i b\cdot\nabla \p_3^{3}\omega_k + \p_3^{3} H_k b\cdot\nabla \p_3^{3}\omega_i \\
			=\; &  \big( \p_3^{3}H_i \p_3^{3}H_k \big)_t  +\p_3^{3}H_i \p_3^{3} (u\cdot\nabla H_k) + \p_3^{3}H_k \p_3^{3} (u\cdot\nabla H_i)\\
			& + \p_3^{3}H_i\p_3^{3}(\sum_{p=1}^3 \nabla u_p \times \p_p b_k) +  \p_3^{3}H_k\p_3^{3}(\sum_{p=1}^3 \nabla u_p \times \p_p b_i)\\
			& - \p_3^{3}H_i\p_3^{3}( \sum_{p=1}^3 \nabla b_p\times \p_p u_k )-\p_3^{3}H_k\p_3^{3}( \sum_{p=1}^3 \nabla b_p\times \p_p u_i ) \\
			& -\p_3^{3}H_i  \sum_{l = 1}^{3} \mathcal{C}_{3}^l\p_3^l b\cdot\nabla \p_3^{3-l}\omega_k -\p_3^{3}H_k  \sum_{l = 1}^{3} \mathcal{C}_{3}^l\p_3^l b\cdot\nabla \p_3^{3-l}\omega_i\\
			& - \p_3^{3}H_i\p_3^{3}\Delta_h H_k -\p_3^{3}H_k\p_3^{3}\Delta_h H_i - \p_3^{3}H_i\p_3^{3} \p_2 \omega_k -\p_3^{3}H_k\p_3^{3} \p_2 \omega_i .
		\end{split}
	\end{equation}
We can then replace the nonlinear terms in $J_{411}$ by the right hand side of \eqref{6eq7}.
This complicated substitution generates many more terms, but this important process converts some of seemingly
impossible terms into other terms that can be bounded suitably. This is what we would call artificial cancellation through substitutions.  Next we give the estimate for the terms corresponding to the right hand side of \eqref{6eq7}. The first term we come across is
$$
K_1 =\int_{\mathbb{R}^3} \big( \p_3^{3}H_i \p_3^{3}H_k \big)_t (b_j | \p_2 u_j) \; dx.
$$
Using integration by parts and invoking the equation of $b$ in \eqref{mhd}, we can rewrite the term containing $b_j$ as follows,
	\begin{equation}\nonumber
		\begin{split}
			K_1 = & \; \frac{d}{dt}\int_{\mathbb{R}^3} \big( \p_3^{3}H_i \p_3^{3}H_k \big) b_j \; dx - \int_{\mathbb{R}^3} \big( \p_3^{3}H_i \p_3^{3}H_k \big) \p_t b_j \; dx\\
			= & \; \frac{d}{dt}\int_{\mathbb{R}^3} \big( \p_3^{3}H_i \p_3^{3}H_k \big) b_j \; dx \\
			& - \int_{\mathbb{R}^3} \big( \p_3^{3}H_i \p_3^{3}H_k \big) (-u\cdot \nabla b_j + \Delta_h b_j + b\cdot \nabla u_j + \p_2 u_j) \; dx.
		\end{split}
	\end{equation}
By Lemma \ref{lem-anisotropic-est}, the last line in the equality above can be bounded by
	\begin{equation}\nonumber
		\begin{split}
			&\|\p_3^{3}H\|_{L^2}\|\p_1 \p_3^{3}H\|_{L^2}\|u\|_{H^2}^\frac{1}{2} \|\p_2 u\|_{H^2}^\frac{1}{2}\|b\|_{H^2}^\frac{1}{2} \|\p_2 b\|_{H^2}^\frac{1}{2}
			 \\
			& + \|\p_3^{3}H\|_{L^2}^\frac{1}{2}\|{\bf \p_1} \p_3^{3}H\|_{L^2}^\frac{1}{2}\|\p_3^{3}H\|_{L^2}^\frac{1}{2}\|{\bf \p_2} \p_3^{3}H\|_{L^2}^\frac{1}{2}\\
				&\quad \times \Big(\|\Delta_h b\|_{L^2}^\frac{1}{2}\|{\bf \p_3} \Delta_h b\|_{L^2}^\frac{1}{2}
+\|\p_2 u\|_{L^2}^\frac{1}{2}\|{\bf \p_3} \p_2 u\|_{L^2}^\frac{1}{2}\Big).
		\end{split}
	\end{equation}
The estimate for the term containing  $\p_2 u_j$ is similar.
	We move on to the second part
	$$
	K_2 = \int_{\mathbb{R}^3} \big[ \p_3^{3}H_i \p_3^{3} (u\cdot\nabla H_k) + \p_3^{3}H_k \p_3^{3} (u\cdot\nabla H_i) \big] (b_j | \p_2 u_j) \; dx.
	$$
It can be handled similarly as $J_{11}$.  Due to its symmetric property,
	\begin{equation}\nonumber
		\begin{split}
			&\int_{\mathbb{R}^3} \Big[ \p_3^{3}H_i u\cdot\nabla \p_3^{3} H_k + \p_3^{3}H_k u\cdot\nabla \p_3^{3}H_i \Big] (b_j|\p_2 u_j )dx \\
			=& \int_{\mathbb{R}^3} u\cdot\nabla \big( \p_3^{3}H_i\p_3^{3}H_k \big)(b_j|\p_2 u_j )dx \\
			=& - \int_{\mathbb{R}^3} \p_3^{3}H_i\p_3^{3}H_k u\cdot\nabla (b_j|\p_2 u_j) dx \\
			\lesssim & \; \| {\bf \p_1} H\|_{H^{3}} \|H\|_{H^{3}} \| {\bf \p_2} u\|_{H^{1}}^{\frac{1}{2}} \|u\|_{H^1}^{\frac{1}{2}} \| {\bf \p_2} (b|\p_2 u)\|_{H^{2}}^{\frac{1}{2}} \| (b|\p_2 u)\|_{H^{2}}^{\frac{1}{2}}
		\end{split}
	\end{equation}
	and
	\begin{equation}\nonumber
		\begin{split}
			&\sum_{l = 1}^3 \mathcal{C}_3^l \int_{\mathbb{R}^3} \Big[ \p_3^{3}H_i \p_3^l u\cdot\nabla \p_3^{3-l} H_k + \p_3^{3}H_k \p_3^l u\cdot\nabla \p_3^{3-l}H_i \Big] (b_j|\p_2 u_j )dx \\
			\lesssim & \; \|\p_3^3 H\|_{L^2}^\frac{1}{2}\|{\bf \p_2}\p_3^3 H\|_{L^2}^\frac{1}{2} \|H\|_{H^3}^\frac{1}{2}\|{\bf \p_2}H\|_{H^3}^\frac{1}{2} \|u\|_{H^4}^\frac{1}{2}\|{\bf \p_1}u\|_{H^4}^\frac{1}{2} \|(b|\p_2 u)\|_{H^1}^\frac{1}{2}\|{\bf \p_1}(b|\p_2 u)\|_{H^1}^\frac{1}{2}.
		\end{split}
	\end{equation}
The next term $K_3$ contains six parts of (\ref{6eq7}) and is defined as follows.
	\begin{equation}\nonumber
		\begin{split}
			K_3 = & \int_{\mathbb{R}^3} \Big(\p_3^{3}H_i\p_3^{3}(\sum_{p=1}^3 \nabla u_p \times \p_p b_k) +  \p_3^{3}H_k\p_3^{3}(\sum_{p=1}^3 \nabla u_p \times \p_p b_i)\\
			& - \p_3^{3}H_i\p_3^{3}( \sum_{p=1}^3 \nabla b_p\times \p_p u_k )-\p_3^{3}H_k\p_3^{3}( \sum_{p=1}^3 \nabla b_p\times \p_p u_i ) \\
			& -\p_3^{3}H_i  \sum_{l = 1}^{3} \mathcal{C}_{3}^l\p_3^l b\cdot\nabla \p_3^{3-l}\omega_k -\p_3^{3}H_k  \sum_{l = 1}^{3} \mathcal{C}_{3}^l\p_3^l b\cdot\nabla \p_3^{3-l}\omega_i \Big) \cdot (b_j | \p_2 u_j) \; dx.
		\end{split}
	\end{equation}
By Lemma \ref{lem-anisotropic-est}, it's easy to derive
	\begin{equation}\nonumber
	\begin{split}
		K_3 \lesssim & \; \|H\|_{H^{3}}^\frac{1}{2}\|{\bf \p_1} H\|_{H^{3}}^\frac{1}{2}
		\big(\|u\|_{{ H^{4}}}^\frac{1}{2}\|{\bf \p_1} u\|_{{H^{4}}}^\frac{1}{2}
		{\|b\|_{H^{3}}^\frac{1}{2} \|{\bf \p_2}b\|_{H^{3}}^\frac{1}{2}}\\
		& + \|b\|_{ {H^{4}} }^\frac{1}{2}\|{\bf \p_1} b\|_{ {H^{4}} }^\frac{1}{2}
		{\|u\|_{H^{3}}^\frac{1}{2} \|{\bf \p_2}u\|_{H^{3}}^\frac{1}{2}} \big)
		\cdot \|(b|\p_2u)\|_{H^1}^\frac{1}{2}\|{\bf \p_2}(b|\p_2u)\|_{H^1}^\frac{1}{2}.
	\end{split}
\end{equation}
The last four parts of (\ref{6eq7}) are included in $K_4$,
	\begin{equation}\nonumber
		\begin{split}
			K_4 {=} & \int_{\mathbb{R}^3} \Big[   - \p_3^{3}H_i\p_3^{3}\Delta_h H_k -\p_3^{3}H_k\p_3^{3}\Delta_h H_i \Big] (b_j|\p_2 u_j )dx \\
			& + \int_{\mathbb{R}^3} \Big[- \p_3^{3}H_i\p_3^{3} \p_2 \omega_k -\p_3^{3}H_k\p_3^{3} \p_2 \omega_i \Big] (b_j|\p_2 u_j )dx \\
			\lesssim &\;  \|\na_h H\|_{H^{3}}^2 \|(b |\p_2 u)\|_{H^2} + \|H\|_{H^{3}}\|\na_h H\|_{H^{3}} \|\na_h(b|\p_2 u)\|_{H^2} \\
			& + \| \p_2 \p_3^{3}H \|_{L^2} \| \p_3^{3} \omega  \|_{L^2}^{\frac{1}{2}} \| {\bf \p_1} \p_3^{3} \omega  \|_{L^2}^{\frac{1}{2}} \| (b|\p_2 u )\|_{H^1}^{\frac{1}{2}} \| {\bf \p_2}(b|\p_2 u )\|_{H^1}^{\frac{1}{2}} \\
			& + \|\p_3^{3}H \|_{L^2}^{\frac{1}{2}} \|{\bf \p_2}\p_3^{3}H \|_{L^2}^{\frac{1}{2}}  \| \p_3^{3} \omega  \|_{L^2}^{\frac{1}{2}} \| {\bf \p_1} \p_3^{3} \omega  \|_{L^2}^{\frac{1}{2}}   \| \p_2(b|\p_2 u ) \|_{L^2}^{\frac{1}{2}} \| {\bf \p_3}\p_2(b|\p_2 u ) \|_{L^2}^{\frac{1}{2}}.
		\end{split}
	\end{equation}
We have estimated all the terms corresponding to the right hand side of \eqref{6eq7} and thus obtained a suitable upper bound for $J_{411}$ in (\ref{6eq3}). Integrating the upper bounds on $K_1$ through $K_4$ in time yields
	\begin{equation}\nonumber
		\int_0^t |J_{411}| d\tau \lesssim {\mathscr{E}^{\frac{3}{2}}(0)}+\mathscr{E}^{\frac{3}{2}}(t) + \mathscr{E}^2(t).
	\end{equation}
Together with the estimate (\ref{j412b}) for $J_{412}$, we conclude that
	\begin{equation}\nonumber
	\int_0^t |J_{41}| d\tau \lesssim { \mathscr{E}^{\frac{3}{2}}(0)}+ \mathscr{E}^{\frac{3}{2}}(t) + \mathscr{E}^2(t).
\end{equation}
$J_{42}$ has been estimated before in (\ref{j42b}). Thus,
		\begin{equation}\nonumber
		\int_0^t |J_{4}| d\tau \lesssim { \mathscr{E}^{\frac{3}{2}}(0)}+ \mathscr{E}^{\frac{3}{2}}(t) + \mathscr{E}^2(t).
		\end{equation}
We deal with the last term,
	\begin{equation}\nonumber
		\begin{split}
			J_5 = \int_{\mathbb{R}^3} \big[\p_3^3 \omega_i \p_3^3 (H \cdot \nabla b_k - \p_2 H_k) + \p_3^3 \omega_k \p_3^3 (H \cdot \nabla b_i - \p_2 H_i) \big] (b_j | \p_2 u_j) \; dx.
		\end{split}
	\end{equation}
By Lemma \ref{lem-anisotropic-est},
	\begin{equation}\nonumber
		\begin{split}
			J_5 \lesssim &\|\omega\|_{H^3}^\frac{1}{2}\|{\bf \p_1} \omega\|_{H^3}^\frac{1}{2} \|b\|_{H^4}^\frac{1}{2}\|{\bf \p_1} b\|_{H^4}^\frac{1}{2}\|b\|_{H^3}^\frac{1}{2}\|{\bf \p_2} b\|_{H^3}^\frac{1}{2}\|(b|\p_2 u)\|_{H^1}^\frac{1}{2}\|{\bf \p_2} (b|\p_2 u)\|_{H^1}^\frac{1}{2}\\
&+{\|\p_3^3\omega\|_{L^2}^{\frac {1}{2}}\|{\bf\p_1 }\p_3^3\omega\|_{L^2}^{\frac {1}{2}} \|\p_3^3\p_2H\|_{L^2} \|(b|\p_2 u)\|_{H^1}^\frac{1}{2}\|{\bf \p_2} (b|\p_2 u)\|_{H^1}^\frac{1}{2}}.
\end{split}
\end{equation}
Integrating in time yields
\begin{equation}\nonumber
	\int_0^t |J_{5}| d\tau \lesssim  \mathscr{E}^{\frac{3}{2}}(t) + \mathscr{E}^2(t).
\end{equation}
Integrating (\ref{J12345}) in time over $[0, t]$ and combining all the bounds for $J_1$ through $J_5$, we are led to the conclusion of Lemma \ref{lem-assistant}.
\end{proof}

\vskip .1in
We are now ready to prove Proposition \ref{lem-important}.

\begin{proof}[Proof of Proposition \ref{lem-important}] Recall that the goal here is to bound the interaction terms $\mathcal{W}(t)$ defined by
$$ \mathcal{W}^{ijk}(t) \triangleq  \int_{\mathbb{R}^3} \p_3^{3}\omega_i \p_2u_j \p_3^{3}\omega_k \; dx, \quad i,j,k \in \{1,2,3\}.$$
We replace $\p_2 u_j$ by the equation of $b_j$ in (\ref{mhd}), there is
\begin{equation*}\label{W}
\begin{split}
\mathcal{W}^{ijk}(t)=\; & \int_{\mathbb{R}^3} \p_3^{3}\omega_i \Big[\p_t b_j +u\cdot\nabla b_j -\Delta_h b_j -b\cdot\nabla u_j\Big] \p_3^{3}\omega_k \; dx \\
  = \; & \frac{d}{dt} \int_{\mathbb{R}^3} \p_3^{3}\omega_i b_j \p_3^{3}\omega_k dx- \int_{\mathbb{R}^3} b_j \p_t(\p_3^{3}\omega_i \p_3^{3}\omega_k)dx \\
  &+ \int_{\mathbb{R}^3} \p_3^{3}\omega_i u\cdot\nabla b_j \p_3^{3}\omega_kdx - \int_{\mathbb{R}^3} \p_3^{3}\omega_i b\cdot\nabla u_j \p_3^{3}\omega_kdx \\
  & -\int_{\mathbb{R}^3} \p_3^{3}\omega_i \Delta_h b_j \p_3^{3}\omega_kdx \\
  =\; & \mathcal{W}_1^{ijk}+\mathcal{W}_2^{ijk}+\mathcal{W}_3^{ijk}+\mathcal{W}_4^{ijk}+\mathcal{W}_5^{ijk} .
\end{split}
\end{equation*}
By H\"{o}lder's inequality,
\begin{equation}\nonumber
\int_0^t W_1^{ijk}(\tau) \;d\tau \lesssim \mathscr{E}^\frac{3}{2}(0) +  \mathscr{E}^\frac{3}{2}(t).
\end{equation}
$\mathcal{W}_2^{ijk}$ is an extremely difficult term. Fortunately,
we have bounded it in Lemma  \ref{lem-assistant},
\begin{equation}\nonumber
\left|\int_0^t \mathcal{W}_2^{ijk}(\tau) d\tau \right| \lesssim {\mathscr{E}^{\frac{3}{2}}(0)}+ \mathscr{E}^{\frac{3}{2}}(t) + \mathscr{E}^2(t).
\end{equation}
By Lemma \ref{lem-anisotropic-est},
\begin{equation}\nonumber
\begin{split}
\mathcal{W}_3^{ijk} \lesssim & \; \|\p_3^{3}\omega_i \|_{L^2}^{\frac{1}{2}}\|{\bf \p_1}\p_3^{3}\omega_i \|_{L^2}^{\frac{1}{2}} \|\p_3^{3}\omega_k \|_{L^2}^{\frac{1}{2}}\|{\bf \p_1}\p_3^{3}\omega_k \|_{L^2}^{\frac{1}{2}}\\
&\cdot {\|u\|_{H^1}^{\frac{1}{2}} \|{\bf \p_2}u\|_{H^1}^{\frac{1}{2}}
	\| \nabla b_j\|_{H^1}^{\frac{1}{2}} \| {\bf \p_2}\nabla b_j\|_{H^1}^{\frac{1}{2}} }	
\end{split}
\end{equation}
and
\begin{equation}\nonumber
\begin{split}
\mathcal{W}_4^{ijk} \lesssim & \; \|\p_3^{3}\omega_i \|_{L^2}^{\frac{1}{2}}\|{\bf \p_1}\p_3^{3}\omega_i \|_{L^2}^{\frac{1}{2}} \|\p_3^{3}\omega_k \|_{L^2}^{\frac{1}{2}}\|{\bf \p_1}\p_3^{3}\omega_k \|_{L^2}^{\frac{1}{2}} \\
&\cdot {\|b\|_{H^1}^{\frac{1}{2}} \|{\bf \p_2}b\|_{H^1}^{\frac{1}{2}}
	\| \nabla u_j\|_{H^1}^{\frac{1}{2}} \| {\bf \p_2}\nabla u_j\|_{H^1}^{\frac{1}{2}} }.
\end{split}
\end{equation}
Integrating in time yields
\begin{equation}\nonumber
	\int_0^t |\mathcal{W}_3^{ijk}(\tau)| d\tau, \,\, \int_0^t |\mathcal{W}_4^{ijk}(\tau)| d\tau\lesssim   \mathscr{E}^2(t).
\end{equation}
We divide $\mathcal{W}_5^{ijk}$ into two parts,
\begin{equation}\nonumber
\begin{split}
  \mathcal{W}_5^{ijk} = & \int_{\mathbb{R}^3} \p_1 b_j \p_1 (\p_3^{3}\omega_i \p_3^{3}\omega_k)dx - \int_{\mathbb{R}^3} \p_2^2 b_j \p_3^{3}\omega_i\p_3^{3}\omega_kdx \\
  = & \; \mathcal{W}_{51}^{ijk} + \mathcal{W}_{52}^{ijk}.
\end{split}
\end{equation}
The estimate for $\mathcal{W}_{51}^{ijk}$ is easy,
\begin{equation}\nonumber
\begin{split}
\mathcal{W}_{51}^{ijk} \lesssim \|\p_1 b\|_{H^2} \|\p_1 \omega\|_{H^{3}} \|\omega\|_{H^{3}}.
\end{split}
\end{equation}
For $\mathcal{W}_{52}^{ijk}$, we replace $\p_2 b_j$ via the equation of {$u_j$} in \eqref{mhd},
\begin{equation}\nonumber
  \mathcal{W}_{52}^{ijk} = - \int_{\mathbb{R}^3} \p_2 \big( \p_tu_j +u\cdot\nabla u_j -\p_1^2u_j -b\cdot\nabla b_j + \p_j P\big) \p_3^{3}\omega_i\p_3^{3}\omega_kdx .
\end{equation}
By Lemma \ref{lem-anisotropic-est} and Sobolev's inequality,
\begin{equation}\nonumber
	\begin{split}
		& - \int_{\mathbb{R}^3} \p_2 \big(u\cdot\nabla u_j -\p_1^2u_j -b\cdot\nabla b_j + \nabla_j P\big) \p_3^{3}\omega_i\p_3^{3}\omega_kdx \\
		\lesssim & \; \|\p_1 u\|_{H^4} \|\p_1 \omega\|_{H^{3}} \|\omega\|_{H^3} + \|\p_3^{3}\omega \|_{L^2}\|{\bf \p_1}\p_3^{3}\omega \|_{L^2}  \big(\|u\|_{H^3} \|{\bf \p_2}u\|_{H^3}  + \|b\|_{H^3}\|{\bf \p_2}b\|_{H^3}\big)\\
		& + \|\p_2 \nabla P\|_{H^1}^\frac{1}{2} \|{\bf \p_2} \p_2 \nabla P\|_{H^1}^\frac{1}{2}\|\p_3^{3}\omega \|_{L^2}^\frac{3}{2}\|{\bf \p_1}\p_3^{3}\omega \|_{L^2}^\frac{1}{2}.
	\end{split}
\end{equation}
Taking the divergence of the velocity equation in (\ref{mhd}) yields a representation of the pressure  $P$,
$$
P = \sum_{i,j=1}^3 (-\Delta)^{-1} \p_i \p_j (u_i u_j - b_i b_j).
$$
Using the fact that Riesz operators are bounded on $L^q$ for $1< q < \infty$ and the third inequality in Lemma \ref{lem-anisotropic-est}, we have
\begin{equation}\nonumber
	\begin{split}
		\|\p_2 P\|_{L^2} \lesssim & \sum_{v = u, b} \|\p_2(v \otimes v)\|_{L^2} \lesssim \sum_{v = u, b}\|\p_2 v\|_{L^2}^\frac{1}{2} \|{\bf \p_1} \p_2 v\|_{L^2}^\frac{1}{2}\|v\|_{H^1}^\frac{1}{2} \|{\bf \p_2} v\|_{H^1}^\frac{1}{2}, \\
		\|\p_2^2 P\|_{L^2} \lesssim & \sum_{v = u, b} \|\p_2^2(v \otimes v)\|_{L^2} \lesssim \sum_{v = u, b}\big(\|\p_2 v\|_{L^4}^2+  \|\p_2^2 v\|_{L^2}^\frac{1}{2}\|{\bf \p_1} \p_2^2 v\|_{L^2}^\frac{1}{2}
		\|v\|_{H^1}^\frac{1}{2}\|{\bf \p_2} v\|_{H^1}^\frac{1}{2}\big).
	\end{split}
\end{equation}
Hence,
\begin{equation}\nonumber
  \begin{split}
    & - \int_{\mathbb{R}^3} \p_2 \big(u\cdot\nabla u_j -\p_1^2u_j -b\cdot\nabla b_j + \nabla_j P\big) \p_3^{3}\omega_i\p_3^{3}\omega_kdx \\
    \lesssim & \; \|\p_1 u\|_{H^4} \|\p_1 \omega\|_{H^{3}} \|\omega\|_{H^3} + \|\p_3^{3}\omega \|_{L^2}\|{\bf \p_1}\p_3^{3}\omega \|_{L^2}  \big(\|u\|_{H^3} \|{\bf \p_2}u\|_{H^3}  + \|b\|_{H^3}\|{\bf \p_2}b\|_{H^3}\big)\\
    & + \sum_{v = u, b} \big(\|\p_2 v\|_{H^2}^\frac{1}{2} \|{\bf \p_1} \p_2 v\|_{H^2}^\frac{1}{2}\|v\|_{H^3}^\frac{1}{2} \|{\bf \p_2} v\|_{H^3}^\frac{1}{2}\big)^\frac{1}{2} \\
    & \cdot \big(\|\p_2 v\|_{H^3}^2+  \|\p_2^2 v\|_{H^2}^\frac{1}{2}\|{\bf \p_1} \p_2^2 v\|_{H^2}^\frac{1}{2}
\|v\|_{H^3}^\frac{1}{2}\|{\bf \p_2} v\|_{H^3}^\frac{1}{2}\big)^\frac{1}{2} \|\p_3^{3}\omega \|_{L^2}^\frac{3}{2}\|{\bf \p_1}\p_3^{3}\omega \|_{L^2}^\frac{1}{2}.
  \end{split}
\end{equation}
This implies that we only have one term left in the estimate for $W_{52}^{ijk}$, namely
$$- \int_{\mathbb{R}^3} \p_2\p_tu_j\;  \p_3^{3}\omega_i\p_3^{3}\omega_kdx.$$
We shift the time derivative via the equation
$$-\int_{\mathbb{R}^3} \p_2\p_t u_j \p_3^{3}\omega_i\p_3^{3}\omega_kdx = -\frac{d}{dt}\int_{\mathbb{R}^3} \p_2 u_j \p_3^{3}\omega_i\p_3^{3}\omega_kdx + \int_{\mathbb{R}^3} \p_2 u_j \p_t(\p_3^{3}\omega_i\p_3^{3}\omega_k)dx.$$
The last term above is a difficult term and is bounded via Lemma \ref{lem-assistant}. Integrating in time and invoking
the upper bounds, we find
$$
	\int_0^t |\mathcal{W}_5^{ijk}(\tau)| d\tau  \lesssim \mathscr{E}^\frac{3}{2}(0) +  \mathscr{E}^\frac{3}{2}(t)
	+ \mathscr{E}^2(t).
$$
This completes the proof of Proposition \ref{lem-important}.
\end{proof}

\vskip .3in
\section{Proof of Theorem \ref{main}}
\label{proof-thm}

This section completes the proof of Theorem \ref{main}, which is achieved by applying the bootstrapping argument
to the energy inequalities obtained in the previous sections.

\begin{proof}[Proof of Theorem \ref{main}] As aforementioned in the introduction, the local well-posedness
of (\ref{mhd}) in $H^4$ follows from a standard procedure (see, e.g., \cite{MaBe}) and our
attention is exclusively focused on the global bound {of $H^4$-norms}. This is accomplished
by the bootstrapping argument. The key components are the two energy inequalities established previously
in Sections \ref{e2b} and \ref{proof-high},
\begin{align}
	&	{\mathscr{E}_{2}(t)}\le C_1\, \mathscr{E}(0) + C_2\,\mathscr{E}_1(t) +  C_3\,\mathscr{E}_1^{\frac32}(t)
	+ C_4 \,\mathscr{E}_2^{\frac32}(t), \label{ineq1}\\
	&{\mathscr{E}_{1}(t)} \le C_5\, \mathscr{E}(0) + C_6 \,\mathscr{E}^{\frac32}(0) +  C_7\, \mathscr{E}^{\frac32}(t) + C_8\, \mathscr{E}^2 (t), \label{ineq2}
\end{align}
where $\mathscr{E} = \mathscr{E}_{1} +\mathscr{E}_{2}$.  Adding (\ref{ineq2}) to $1/(2C_2)$ of (\ref{ineq1}) yields, for a constant $C_0>0$ and for any $t>0$,
\begin{equation}\label{en}
		{\mathscr{E} (t)}\le C_0\, \mathscr{E}(0) +C_0 \,\mathscr{E}^{\frac32}(0) + C_0 \,\mathscr{E}^{\frac32}(t)
		+ C_0\, \mathscr{E}^2 (t).
\end{equation}	
Without loss of generality, we assume $C_0\ge 1$. Applying a bootstrapping argument to (\ref{en}) then implies that, if $\|(u_0, b_0)\|_{H^4}$ is
sufficiently small, say
\begin{equation}\label{inc}
\mathscr{E}(0) \le \frac1{128 C^3_0}\quad \mbox{or}\quad  \|(u_0, b_0)\|_{H^4} \le \epsilon :=\frac1{\sqrt{128 C_0^3}},
\end{equation}
then, for any $t>0$,
$$
	\mathscr{E}(t)\le \frac1{32 C^2_0}, \quad \mbox{especially}\quad \|(u(t), b(t))\|_{H^4} \le 2 \sqrt{C_0}\, \epsilon.
$$
In fact, if we make the ansatz that
\begin{equation}\label{ans}
\mathscr{E}(t)\le \frac1{16 C^2_0}
\end{equation}
and insert (\ref{ans}) in (\ref{en}), we  obtain
$$
\mathscr{E}(t) \le C_0\, \mathscr{E}(0) +C_0 \,\mathscr{E}^{\frac32}(0) + \frac12 \mathscr{E}(t),
$$
which, together with (\ref{inc}), implies
\begin{equation}\label{con}
\mathscr{E}(t) \le\frac1{32 C^2_0}.
\end{equation}
The bootstrapping argument then implies that (\ref{con}) actually holds for all $t>0$. This completes the proof of Theorem \ref{main}.
\end{proof}

\vskip .3in
\section*{Acknowledgement}

Lin is partially supported by NSFC under Grant 11701049. Wu is partially supported by the National Science Foundation of the
United States under grant DMS 2104682 and the AT\&T Foundation at Oklahoma State
University. Zhu is partially supported by NSFC under Grant 11801175.


\vskip .4in
\begin{thebibliography}{89}

\bibitem{AZ} H. Abidi and P. Zhang,
 On the global solution of 3D MHD system with initial data near
equilibrium, {\it Commun. Pure Appl. Math. \bf 70} (2017),  1509--1561.

\bibitem{AMS} A. Alemany, R. Moreau, P.-L. Sulem and U. Frisch, Influence of an external magnetic
field on homogeneous MHD turbulence, {\it  J. M\'{e}c. \bf 18} (1979), 277-313.


\bibitem{Alex} A. Alexakis, Two-dimensional behavior of three-dimensional magnetohydrodynamic flow with
a strong guiding field, {\it Phys. Rev. E \bf 84} (2011), 056330.

\bibitem{Alf} H.  Alfv\'{e}n, Existence of electromagnetic-hydrodynamic waves, {\it  Nature \bf 150} (1942), 405--406.

\bibitem{BSS} C. Bardos, C. Sulem and P.L. Sulem, Longtime dynamics of a conductive fluid in the presence of a strong magnetic field, {\it Trans. Am. Math. Soc. \bf 305} (1988), 175--191.

\bibitem{Bis} D. Biskamp, {\it Nonlinear Magnetohydrodynamics}, Cambridge University Press, Cambridge, 1993.

\bibitem{Bur} P. Burattini, O. Zikanov and B. Knaepen, Decay of magnetohydrodynamic turbulence at
low magnetic Reynolds number, {\it J. Fluid Mech. \bf 657} (2010), 502--538.

\bibitem{CaiLei} Y. Cai and Z. Lei, Global well-posedness of the incompressible magnetohydrodynamics, {\it Arch. Ration. Mech. Anal. \bf 2018} (228), No.3, 969--993.

\bibitem{CaoWu} C. Cao and J. Wu, Global regularity for the 2D MHD equations with mixed partial dissipation and magnetic diffusion, {\it Adv.  Math. \bf  226} (2011), 1803-1822.



\bibitem{CaoWuYuan} C. Cao, J. Wu and B. Yuan, The 2D incompressible
magnetohydrodynamics equations with only magnetic diffusion, {\it SIAM J. Math. Anal.} {\bf 46} (2014),  588--602.

\bibitem{CDGG} J. Chemin, B. Desjardins, I. Gallagher and E. Grenier, {\it Mathematical Geophysics.
	An introduction to rotating fluids and the Navier-Stokes equations}, Oxford Lecture
    Series in Mathematics and its Applications, Vol. 32, The Clarendon Press/Oxford University Press, Oxford, 2006.

\bibitem{CMZ} Q. Chen, C. Miao and Z. Zhang, On the regularity criterion of weak solution for the 3D viscous magneto-hydrodynamics equations, {\it Comm. Math. Phys. \bf 284} (2008),  919-930.

\bibitem{CLi1}  I.  Craig and Y. Litvinenko,	Wave energy dissipation by anisotropic viscosity in magnetic X-points,
{\it Astrophysical J. \bf 667} (2007), 1235-1242.

\bibitem{CLi2} I.  Craig and Y. Litvinenko, Anisotropic viscous dissipation in three-dimensional magnetic merging solutions, {\it Astronomy \verb'&' Astrophysics \bf 501} (2009), 755-760.

\bibitem{Davi0} P.A. Davidson, Magnetic damping of jets and vortices, {\it J. Fluid Mech. \bf 299} (1995), 153--186.

\bibitem{Davi1} P.A. Davidson,  The role of angular momentum in the magnetic damping of turbulence,
{\it J. Fluid Mech. \bf 336} (1997), 123-150.

\bibitem{Davi} P.A. Davidson, {\it An Introduction to Magnetohydrodynamics},  Cambridge University Press, Cambridge, England, 2001.


\bibitem{DongLiWu1} B. Dong, Y. Jia, J. Li and J. Wu, Global regularity and time decay for the 2D magnetohydrodynamic equations with fractional dissipation and partial magnetic diffusion, {\it J. Math. Fluid Mechanics \bf 20} (2018), No.4, 1541--1565.

\bibitem{DongLiWu0} B. Dong, J. Li and J. Wu, Global regularity for the 2D MHD equations with partial hyperresistivity, {\it International Math Research Notices}, 2018,  rnx240, https://doi.org/10.1093  /imrn/rnx240.

\bibitem{Gall} B. Gallet,  M. Berhanu and N. Mordant, Influence of an external
magnetic field on forced turbulence in a swirling flow of liquid metal,
{\it Phys. Fluids \bf 21} (2009), 085107.

\bibitem{Gall2}  B. Gallet and C.R. Doering, Exact two-dimensionalization of low-magnetic-Reynolds-number flows subject to a strong magnetic field, {\it J. Fluid Mech. \bf 773} (2015), 154--177.


\bibitem{FNZ} J. Fan, H. Malaikah, S. Monaquel, G. Nakamura and  Y. Zhou, Global Cauchy
problem of 2D generalized MHD equations, {\it Monatsh. Math. \bf 175} (2014),  127-131.

\bibitem{Fefferman1}C.L. Fefferman, D.S. McCormick, J.C. Robinson and J.L. Rodrigo, Higher order commutator estimates and local existence for the non-resistive MHD equations and related models, {\it J. Funct. Anal. \bf 267} (2014), No. 4, 1035--1056.

\bibitem{Fefferman2}C.L. Fefferman, D.S. McCormick, J.C. Robinson and J.L. Rodrigo, Local existence for the non-resistive MHD equations in nearly optimal Sobolev spaces, {\it Arch. Ration. Mech. Anal. \bf 223} (2017),  677--691.

\bibitem{HeXuYu}L. He, L. Xu and P. Yu, On global dynamics of three dimensional magnetohydrodynamics: nonlinear stability
of Alfv\'{e}n waves, {\it Ann. PDE \bf 4} (2018), Art.5, 105 pp.


\bibitem{HuWang} X. Hu and D. Wang, Global existence and large-time behavior of solutions to the three-dimensional equations of compressible magnetohydrodynamic flows, {\it Arch. Ration. Mech. Anal. \bf 197} (2010), 203-238.

\bibitem{HuangLi}X. Huang and J. Li, Serrin-type blowup criterion for viscous, compressible, and heat conducting Navier-Stokes and magnetohydrodynamic flows, {\it Comm. Math. Phys. \bf 324} (2013), 147-171.

\bibitem{If} D. Iftimie, A uniqueness result for the Navier-Stokes equations with vanishing vertical viscosity,
{\it SIAM	J. Math. Anal. \bf 33} (2002), 1483-1493.

\bibitem{JiangARMA} F. Jiang and S. Jiang, On magnetic inhibition theory in non-resistive magnetohydrodynamic fluids, {\it Arch. Ration. Mech. Anal. \bf 233} (2019), 749-798.

\bibitem{JiangJMPA} F. Jiang and S. Jiang, On inhibition of thermal convection instability by a magnetic field under zero resistivity, {\it J. Math. Pures Appl. \bf 141} (2020), 220-265. 	


\bibitem{JNW} Q. Jiu, D. Niu, J. Wu, X. Xu and H. Yu, The 2D magnetohydrodynamic equations with magnetic diffusion, {\it Nonlinearity \bf 28} (2015),  3935-3955.

\bibitem{JiuZhao2} Q. Jiu and J. Zhao, Global regularity of 2D generalized MHD equations with magnetic diffusion, {\it Z. Angew. Math. Phys. \bf 66} (2015),
    677-687.

 \bibitem{LinDu}H. Lin and L. Du, Regularity criteria for incompressible magnetohydrodynamics equations in three dimensions,{\it Nonlinearity}, {\bf 26} (2013), 219-239.

\bibitem{LinZhang1} F. Lin, L. Xu, and  P. Zhang, Global small solutions to 2-D incompressible MHD system, {\it J. Differential Equations} {\bf 259} (2015), 5440-5485.

\bibitem{LinZh} F. Lin and  P. Zhang, Global small solutions to an MHD-type system: the three-dimensional case,
 {\it Comm. Pure Appl. Math. \bf 67} (2014), 531-580.

\bibitem{ZHANG2} Y. Liu and P. Zhang, Global well-posedness of 3-D anisotropic Navier-Stokes system with large vertical viscous coefficient,
{\it J. Funct. Anal. \bf 279} (2020), No.10, 108736, 33 pp.

\bibitem{MaBe} A. Majda and A. Bertozzi,  {\it Vorticity and Incompressible Flow}, Cambridge University Press, 2002.


\bibitem{Pai2}
M. Paicu, \'{E}quation p\'{e}riodique de Navier-Stokes sans viscosit\'{e} dans une direction, {\it Comm. Part. Differ.
	Eqs. \bf 30} (2005),  1107-1140.


\bibitem{ZHANG1}
M. Paicu and P. Zhang, Global strong solutions to 3-D Navier-Stokes system with strong dissipation in one direction,
{\it Sci. China Math. \bf 62} (2019), No.6, 1175-1204.


\bibitem{PanZhouZhu} R. Pan, Y. Zhou and Y. Zhu,  Global classical solutions of three dimensional viscous MHD system without magnetic diffusion on periodic boxes, {\it Arch. Ration. Mech. Anal.  \bf  227} (2018), No.2, 637--662.

\bibitem{Ped} J.  Pedlosky {\it Geophysical Fluid Dynamics}, 2nd Edition, Springer-Verlag,
Berlin Heidelberg-New York, 1987.


\bibitem{Pri} E. Priest and T. Forbes, {\it Magnetic Reconnection, MHD Theory and Applications}, Cambridge University Press, Cambridge, 2000.

\bibitem{Ren} X. Ren, J. Wu, Z. Xiang and Z. Zhang, Global existence and decay of smooth solution for the 2-D MHD equations without magnetic diffusion, {\it J. Functional Analysis \bf 267} (2014),  503--541.

\bibitem{Ren2} X. Ren, Z. Xiang and Z. Zhang, Global well-posedness for the 2D MHD equations without magnetic diffusion in a strip domain, {\it Nonlinearity \bf 29} (2016), No.4, 1257--1291.

\bibitem{Tan} Z. Tan and Y. Wang, Global well-posedness of an initial-boundary value problem for viscous non-resistive MHD systems,  {\it SIAM J. Math. Anal. \bf 50} (2018), No.1, 1432--1470.

\bibitem{Tao} T. Tao, {\em Nonlinear Dispersive Equations: Local and Global Analysis},
CBMS regional conference series in mathematics, 2006.


\bibitem{WeiZ} D. Wei and Z. Zhang, Global well-posedness of the MHD equations in a homogeneous magnetic field, {\it Anal. PDE \bf 10} (2017), No.6,  1361--1406.

\bibitem{WuMHD2018}J. Wu, The 2D magnetohydrodynamic equations with partial or fractional dissipation,  {\it Lectures on the analysis of nonlinear partial differential equations}, Morningside Lectures on Mathematics, Part 5, MLM5, pp. 283-332, International Press, Somerville, MA, 2018.


\bibitem{WuWu} J. Wu and Y. Wu, Global small solutions to the compressible 2D magnetohydrodynamic system without magnetic diffusion, {\it Adv. Math. \bf 310} (2017), 759--888.

\bibitem{WuWuXu} J. Wu, Y. Wu and X. Xu, Global small solution to the 2D MHD system with a velocity damping term, {\it SIAM J. Math. Anal. \bf 47} (2015), 2630-2656.

\bibitem{WuZhu} J. Wu and Y. Zhu, Gloabal solutions of 3D incompressible MHD system with mixed partial dissipation and magnetic diffusion near an equilibrium, {\it Adv. Math.
	\bf 377} (2021), 107466.

\bibitem{Yam1} K. Yamazaki, On the global well-posedness of N-dimensional generalized MHD system in anisotropic spaces, {\it Adv. Differential Equations \bf 19} (2014),  201--224.

\bibitem{YuanZhao} B. Yuan and J. Zhao, Global regularity of 2D almost resistive MHD equations, {\it Nonlinear Anal. Real World Appl. \bf 41} (2018), 53--65.

\bibitem{Ye} Z. Ye, Remark on the global regularity of 2D MHD equations with almost Laplacian
magnetic diffusion, {\it J. Evol. Equations \bf 18} (2018), No.2,  821-844.


\bibitem{ZT1} T. Zhang, An elementary proof of the global existence and uniqueness
theorem to 2-D incompressible non-resistive MHD system,  arXiv: 1404.5681v1 [math.AP] 23 Apr 2014.


\bibitem{ZT2} T. Zhang, Global solutions to the 2D viscous, non-resistive MHD
system with large background magnetic field, {\it    J.	Differential Equations \bf 260} (2016), 5450-5480.


\bibitem{ZhouZhu} Y. Zhou and Y. Zhu,  Global classical solutions of 2D MHD system with only magnetic diffusion on periodic domain, {\it J. Math. Phys. \bf 59} (2018), No.8, 081505, 12 pp.

\end{thebibliography}
\end{document}